\documentclass[11pt]{article}
\usepackage{amsfonts}
\usepackage{}
\usepackage{mathrsfs}
\usepackage[T1]{fontenc}
\usepackage{latexsym,amssymb,amsmath,amsfonts,amsthm}
\usepackage{graphicx,subfigure}
\usepackage[utf8]{inputenc}
\usepackage{authblk}
\usepackage{mathrsfs}

\newcommand{\R}{{\mathbb R}}

\newcommand{\C}{{\mathbb C}}

\newcommand{\be}{\begin{eqnarray}}
\newcommand{\ben}{\begin{eqnarray*}}
\newcommand{\en}{\end{eqnarray}}
\newcommand{\enn}{\end{eqnarray*}}
\newcommand{\ba}{\backslash}
\newcommand{\pa}{\partial}

\newcommand{\ov}{\overline}

\newcommand{\Om}{\Omega}

\newcommand{\la}{\lambda}

\newcommand{\hx}{\hat{x}}

\newtheorem{theorem}{Theorem}[section]

\newtheorem{remark}[theorem]{Remark}

\title{Data recovery: from limited-aperture to full-aperture}
\author{
Xiaodong Liu\thanks{Institute of Applied Mathematics, Academy of Mathematics and Systems Science, Chinese Academy of Sciences, 100190 Beijing, China.  Email: xdliu@amt.ac.cn}
\quad and\quad
Jiguang Sun\thanks{Department of Mathematical Sciences, Michigan Technological University and College of Mathematical Sciences, University of Electronic Science and Technology of China.
Email: jiguangs@mtu.edu}
}
\begin{document}
\maketitle
%\vspace{.2in}

\begin{abstract}
The inverse scattering problems have been popular for the past thirty years.
While very successful in many cases, progress has lagged when only {\em limited-aperture} measurement is available.
In this paper, we perform some elementary study to recover data that can not be measured directly.
To be precise, we aim at recovering the {\em full-aperture} far field data from  {\em limited-aperture} measurement.
Due to the reciprocity relation, the multi-static response matrix (MSR) has a symmetric structure.
Using the Green's formula and single layer potential, we propose two schemes to recover {\em full-aperture} MSR.
The recovered data is tested by a recently proposed direct sampling method and the factorization method.
The numerical results show the possibility to recover, at least partially, the missing data and consequently improve the reconstruction of the scatterer.

\vspace{.2in} {\bf Keywords:}
inverse scattering, multi-static response matrix, limited-aperture, data recovery.

\vspace{.2in} {\bf AMS subject classifications:}
35P25, 35Q30, 45Q05, 78A46

\end{abstract}

\section{Introduction}
\setcounter{equation}{0}
The inverse scattering theory has been a fast-developing area for the past thirty years. The aim is to detect and identify the
unknown objects using acoustic, electromagnetic, or elastic waves.
Many methods have been proposed, e.g., iterative methods, decomposition methods, the linear sampling method, the factorization method and direct sampling methods
\cite{CC2014, Chen, CCHuang, ColtonKirsch, ColtonMonk, ItoJinZou, Kirsch98, KirschGrinberg, KirschKress, LiuIP17, LiuZhang2015, Potthast2010, Sun2012IP}.
Most of the above algorithms use {\em full-aperture} data, i.e., data of  all the observation directions due to all incident directions.
However, in many cases of practical interest, it is not possible to measure the {\em full-aperture} data, e.g., underground mineral
prospection, mine location in the battlefield, and anti-submarine detection.
Consequently, only {\em limited-aperture} data over a range of angles are available.

Various reconstruction algorithms using {\em limited-aperture} data have been developed \cite{AhnJeonMaPark, BaoLiu, IkehataNiemiSiltanen, Robert1987, KirschGrinberg, L4, MagerBleistein,MagerBleistein1978,Zinn1989}.
Although uniqueness of the inverse problems can be proved in some cases \cite{CK}, the quality of the reconstructions are not satisfactory.
Indeed, {\em limited-aperture} data put forward a severe challenge for all the existing numerical methods.
A typical feature is that the "shadow region" is elongated in down range \cite{L4}.
Physically, the information from the "shadow region" is very weak, especially for high frequency waves \cite{MagerBleistein}.
For two-dimensional problems, the numerical
experiments of the decomposition methods  in \cite{Robert1987,Zinn1989} indicate that satisfactory reconstructions need an aperture not smaller than 180 degrees.

% All the above methods use {\em limited-aperture} data directly for the associated inverse problems.
Other than developing methods using  {\em limited-aperture} data, we take an alternative approach to recover the data that can not be measured directly.
As a consequence, methods using {\em full-aperture} data can be employed.
%While very successful in many cases, progress has lagged in other areas of applications which are forced to rely on different modalities using limited data, e.g., reconstructions
%of buried objects.
%While this problem is of obvious practical significance, because in practice it is impossible to measure the scattered amplitude in all
%directions around the objects.
We take the acoustic scattering by time-harmonic plane waves as the model problem.
The measurement data are only available for limited-aperture observation angles but for all incident directions.
The goal is to recover data for all observation angles.
The case to recover full-aperture data from limited-aperture observation angles due to limited-aperture incident directions will be considered in future.
%The contribution of this work is to introduce some techniques to retrieve full-aperture data based on limited aperture data.

For scattering problems, it is well-known that the full-aperture data can be uniquely determined by the limited-aperture data.
However, because of the severely ill-posed nature of the analytic continuation, it is in general not possible to recover full-aperture data using techniques such as extrapolation \cite{Atkinson}.
We take a different way by seeking an analytic function in a suitable space based on the PDE theory governing the scattering problem. More precisely, we look for kernels of layer potentials that generate the measured data approximately by regularization.
Then these kernels are used to obtain the full-aperture data.

This paper is organized as follows. In Section 2, we briefly introduce the scattering problem of interests and the multi-static response (MSR) matrix, which is the far field pattern.
Due to the reciprocity relation of the far field pattern, the MSR has a symmetry property,
which can be used to recover partial missing data.
In Section \ref{sec2.2}, we propose a technique using the Green's formula to recover the full MSR.
Another recovery technique based on the single layer potential is proposed in subsection \ref{sec2.3}.
Combining these techniques and the symmetry property, a novel algorithm is proposed to recover the {\em full-aperture} MSR.
In Section \ref{NEs}, numerical examples are presented to demonstrate the performance of the data recover techniques.
The recovered data are tested using a direct sampling method and the factorization method.
We draw some conclusions and discuss future works in Section 5.
%We restrict our presentation to the two-dimensional case.
%All the results of this paper remain valid in three dimensions using the corresponding fundamental solution and the radiation condition.

\section{The Multi-static Response Matrix}
\label{sec2}
\setcounter{equation}{0}
Let $k$ be the wave number of a time harmonic wave and $\Om\subset\R^n (n=2,\, 3)$ be a bounded domain with
Lipschitz-boundary $\pa\Om$ such that the exterior $\R^n\ba\ov{\Om}$ is connected.
Let the incident field $u^i$ be given by
\be\label{incidenwave}
u^i(x)\ =\ u^i(x;d) = e^{ikx\cdot d},\quad x\in\R^n\,,
\en
where $d\in S^{n-1}, S^{n-1}:=\{x\in\R^n:|x|=1\}$, denotes the direction of the plane wave.

The scattering problem for an inhomogeneous medium is to find the total field $u=u^i+u^s$ such that
\be
\label{HemEqumedium}\Delta u + k^2 (1+q)u = 0\quad \mbox{in }\R^n,\\
\label{Srcmedium}\lim_{r:=|x|\rightarrow\infty}r^{\frac{n-1}{2}}\left(\frac{\pa u^{s}}{\pa r}-iku^{s}\right) =\,0,
\en
where $q\in L^{\infty}(\R^n)$ such that its imaginary part $\Im (q)\geq 0$ and $q=0$ in $\R^n\ba\ov{\Om}$.
The Sommerfeld radiation condition \eqref{Srcmedium} holds uniformly with respect to all directions $\hx:=x/|x|\in S^{n-1}$.
If the scatterer $\Om$ is impenetrable, the direct scattering problem is to find the total field $u=u^i+u^s$ such that
\be
\label{HemEquobstacle}\Delta u + k^2 u = 0\quad \mbox{in }\R^n\ba\ov{\Om},\\
\label{Bc}\mathcal{B}(u) = 0\quad\mbox{on }\pa\Om,\\
\label{Srcobstacle}\lim_{r:=|x|\rightarrow\infty}r^{\frac{n-1}{2}}\left(\frac{\pa u^{s}}{\pa r}-iku^{s}\right) =\,0,
\en
where $\mathcal{B}$ denotes one of the following three boundary conditions:
\ben
(1)\,\mathcal{B}(u):=u\,\,\mbox{on}\ \pa \Om;\qquad
(2)\,\mathcal{B}(u):=\frac{\pa u}{\pa\nu}\,\,\mbox{on}\ \pa \Om;\qquad
(3)\,\mathcal{B}(u):=\frac{\pa u}{\pa\nu}+\la u\,\,\mbox{on}\ \pa \Om
\enn
corresponding, respectively, to the cases when the scatterer $\Omega$ is sound-soft, sound-hard, and of the impedance type.
Here, $\nu$ is the unit outward normal to $\pa\Om$ and $\la\in L^{\infty}(\pa \Om)$ is the (complex valued) impedance function such that $\Im(\la)\geq0$ almost everywhere on $\pa \Om$.
Uniqueness of the scattering problems \eqref{HemEqumedium}--\eqref{Srcmedium} and \eqref{HemEquobstacle}--\eqref{Srcobstacle}
can be shown with the help of Green's theorem, Rellich's lemma and unique continuation principle, see e.g., \cite{CK}. The proof of existence can be done
by variational approaches (cf. \cite{CK, Mclean} for the Dirichlet boundary condition and \cite{CC2014} for other boundary conditions) or by integral equation methods
(cf. \cite{CK}).

Radiating solutions of the Helmholtz equation have the following asymptotic behavior \cite{KirschGrinberg, LiuIP17}:
\be\label{0asyrep}
u^s(x;d)
=\frac{e^{i\frac{\pi}{4}}}{\sqrt{8k\pi}}\left(e^{-i\frac{\pi}{4}}\sqrt{\frac{k}{2\pi}}\right)^{n-2}\frac{e^{ikr}}{r^{\frac{n-1}{2}}}\left\{u^{\infty}(\hat{x};d)+\mathcal{O}\left(\frac{1}{r}\right)\right\}\quad\mbox{as }\,r:=|x|\rightarrow\infty
\en
uniformly with respect to all directions $\hx:=x/|x|\in S^{n-1}$.
The complex valued function $u^\infty(\hat{x})=u^\infty(\hat{x};d)$ defined on the unit sphere $S^{n-1}$
is known as the scattering amplitude or far-field pattern with $\hat{x}\in S^{n-1}$ denoting the observation direction.

It is well known that the scatterer $\Om$ can be uniquely determined by the far field pattern $u^\infty(\hx,d)$ for all $\hx, d\in S^{n-1}$ \cite{CK}.
% Many direct methods use the full-aperture far field pattern $u^\infty(\hx,d)$ to reconstruct the support of $\Om$.
Due to analyticity, $u^\infty(\hx,d)$ for $(\hx, d) \in S^{n-1} \times S^{n-1}$ is uniquely determined by
 $u^\infty(\hx,d)$ for $(\hx, d) \in S_{0}^{n-1} \times S^{n-1}$ if $S_{0}^{n-1}\subsetneq S^{n-1}$ has a nonempty interior.
Unfortunately, it is practically impossible to obtain $u^\infty(\hx,d)$ on $S^{n-1}$ from $u^\infty(\hx,d)$ on $S_{0}^{n-1}$ using the analytic continuation (see Atkinson \cite{Atkinson}).
% To the author's knowledge, there exists no numerical method for the analytic continuation in our case.

In this paper, we consider the discrete version of $u^\infty(\hx,d)$, i.e., the
 multi-static response (MSR) matrix in $\R^2$.
Let $\theta_i:=(i-1)\pi/m,\,i=1,2,\ldots,2m$,
\ben
d_i:=(\cos\theta_i, \sin\theta_i),\quad i=1,2,\ldots,2m,
\enn
and
\ben
\hat{x}_j:=(\cos\theta_j, \sin\theta_j),\quad j=1,2,\ldots,2m.
\enn
The multi-static response (MSR) matrix $\mathbb{F}_{full}\in \C^{2m\times 2m}$ is defined as
\be\label{MSR}
\mathbb{F}_{full}
:= \left(
    \begin{array}{cccc}
      u_{1,1}^\infty\quad u_{1,2}^\infty\quad \cdots\quad u_{1,2m}^\infty \\
      u_{2,1}^\infty\quad u_{2,2}^\infty\quad \cdots\quad u_{2,2m}^\infty \\
      \vdots\,\qquad \vdots\,\quad \ddots\,\qquad \vdots \\
      u_{2m,1}^\infty\,\, u_{2m,2}^\infty\,\, \cdots\quad u_{2m,2m}^\infty \\
      \end{array}
  \right),
\en
where $u^{\infty}_{i,j}=u^\infty(\hat{x}_j;d_i)$ for $1\leq i, j\leq 2m$ corresponding to $2m$ observation directions $\hat{x}_j$ and $2m$ incident directions $d_i$.
%The response matrix can be regarded as a discrete version of the far field operator.
%The MSR matrix $\mathbb{F}_{full}$ given in \eqref{MSR} is regarded as the far field pattern in {\em full-aperture}.

Assume that the far field pattern can only be measured in a {\em limited-aperture}. In particular,
the measured data are the first $l$ columns of $\mathbb{F}_{full}$
\be\label{MSR-l}
\mathbb{F}^{(l)}_{limit}
:= \left(
    \begin{array}{cccc}
      u_{1,1}^\infty\, u_{1,2}^\infty\, \cdots\, u_{1,l}^\infty \\
      u_{2,1}^\infty\, u_{2,2}^\infty\, \cdots\, u_{2,l}^\infty \\
      \vdots\,\quad \vdots\,\quad \ddots\,\quad \vdots \\
      u_{2m,1}^\infty\, u_{2m,2}^\infty\, \cdots\, u_{2m,l}^\infty \\
      \end{array}
  \right), \quad 1\leq l<2m.
\en
The inverse problem considered in this paper is to firstly recover $\mathbb{F}_{full}$ from $\mathbb{F}^{(l)}_{limit}$, and then reconstruct the scatterer $\Om$ from the recovered $\mathbb{F}_{full}$. Note that $\mathbb{F}_{full}$ is NOT symmetric, i.e., $\mathbb{F}_{full}\neq\mathbb{F}_{full}^T$.
Here and throughout the paper we use the superscript $"T"$ to denote the transpose of a matrix.
We can partition the $2m$-by-$2m$ MSR matrix $\mathbb{F}_{full}$ into a $2$-by-$2$ block matrix
\be\label{Fpartition}
\mathbb{F}_{full}=\left(
             \begin{array}{cc}
               \mathbb{F}_{11} & \mathbb{F}_{12} \\
               \mathbb{F}_{21} & \mathbb{F}_{22} \\
             \end{array}
           \right),
\en
where $\mathbb{F}_{ij}\in \C^{m\times m}, i,j=1,2$. The following theorem is a consequence of the reciprocity relation.

\begin{theorem}\label{blocksymmetric}
$\mathbb{F}_{11}=\mathbb{F}_{22}^{T},\quad \mathbb{F}_{12}=\mathbb{F}_{12}^{T}$ and $\mathbb{F}_{21}=\mathbb{F}_{21}^{T}$.
\end{theorem}
\begin{proof}
Recall that the far field pattern is the same if the direction of the incident field and the observation direction are interchanged \cite{CK}, i.e.,
\be\label{reciprocityrelation}
u^{\infty}(\hat{x},d)= u^{\infty}(-d,-\hat{x}),\quad \mbox{for all}\,\, \hat{x},d\in S^{1}.
\en
For all $u^{\infty}_{i,j}\in \mathbb{F}_{11}$, using the reciprocity relation \eqref{reciprocityrelation}, we have
\ben
u^{\infty}_{i,j}
&=&u^\infty(\hat{x}_j;d_i)\cr
&=&u^\infty(-d_i;-\hat{x}_j)\cr
&=&u^\infty(-(\cos\theta_i,\sin\theta_i);-(\cos\theta_j,\sin\theta_j))\cr
&=&u^\infty((\cos(\theta_i+\pi),\sin(\theta_i+\pi)); (\cos(\theta_j+\pi),\sin(\theta_j+\pi)))\cr
&=&u^\infty((\cos\theta_{i+m},\sin\theta_{i+m}); (\cos\theta_{j+m},\sin\theta_{j+m}))\cr
&=&u^\infty_{j+m,i+m},\quad 1\leq i, j\leq m.
\enn
Thus, we have $\mathbb{F}_{11}=\mathbb{F}_{22}^{T}$.

Similarly, For all $u^{\infty}_{i,j+m}\in \mathbb{F}_{12}$, using the reciprocity relation \eqref{reciprocityrelation} again, we have
\ben
u^{\infty}_{i,j+m}
&=&u^\infty(\hat{x}_{j+m};d_i)\cr
&=&u^\infty(-d_i;-\hat{x}_{j+m})\cr
&=&u^\infty(-(\cos\theta_i,\sin\theta_i);-(\cos\theta_{j+m},\sin\theta_{j+m}))\cr
&=&u^\infty((\cos(\theta_i+\pi),\sin(\theta_i+\pi)); (\cos(\theta_{j+m}+\pi),\sin(\theta_{j+m}+\pi)))\cr
&=&u^\infty((\cos\theta_{i+m},\sin\theta_{i+m}); (\cos\theta_{j},\sin\theta_{j}))\cr
&=&u^\infty_{j,i+m},\quad 1\leq i, j\leq m.
\enn
Thus, we have $\mathbb{F}_{12}=\mathbb{F}_{12}^{T}$. The equality $\mathbb{F}_{21}=\mathbb{F}_{21}^{T}$ can be treated analogously.
\end{proof}
As a direct consequence of Theorem \ref{blocksymmetric}, the following data
\be\label{MSR-2}
\widetilde{\mathbb{F}}^{(l)}_{limit}
:= \left(
    \begin{array}{cccc}
      u_{m+1,l+1}^\infty\, u_{m+1,l+2}^\infty\, \cdots\, u_{m+1,2m}^\infty \\
      u_{m+2,l+1}^\infty\, u_{m+2,l+2}^\infty\, \cdots\, u_{m+2,2m}^\infty \\
      \vdots\,\qquad \vdots\,\qquad \ddots\,\qquad \vdots \\
      u_{m+l,l+1}^\infty\, u_{m+l,l+2}^\infty\, \cdots\, u_{m+l,2m}^\infty \\
      \end{array}
  \right), \quad 1\leq l<2m,
\en
can be obtained directly from $\mathbb{F}^{(l)}_{limit}$. Here, we have set $u_{i,j}^\infty:=u_{i-2m,j}^\infty$ if $i>2m, 1\leq j<2m$.

\begin{remark}
If we set
\ben
\widehat{\mathbb{F}}_{full}:=\left(
             \begin{array}{cc}
               \mathbb{F}_{12} & \mathbb{F}_{11} \\
               \mathbb{F}_{22} & \mathbb{F}_{21} \\
             \end{array}
           \right),
\enn
then $\widehat{\mathbb{F}}_{full}$ is symmetric, i.e., $\widehat{\mathbb{F}}_{full}=\widehat{\mathbb{F}}_{full}^{T}$.
The result also holds for phaseless MSR matrix.
\end{remark}

\section{Data Recover Schemes}\label{DRT}
In this section, we propose two methods to recover $\mathbb{F}_{full}$ from $\mathbb{F}^{(l)}_{limit}$.
The first one is based on the Green's formula. The second one is based on the single layer potential.

\subsection{Method of Green's Formula}\label{sec2.2}
Let $B$ be a bounded domain with connected complement such that $\ov{\Om}\subset B$ and the boundary $\pa B$ is of class $C^2$.
Let $\nu$ denote the unit normal vector to the boundary $\pa B$ directed into the exterior of $B$.
The fundamental solution $\Phi(x,y), x, y \in \R^2, x\neq y,$ of the Helmholtz equation is given by
\be\label{Phi}
\Phi(x,y):=\frac{i}{4}H^{(1)}_0(k|x-y|),
\en
where $H^{(1)}_0$ is the Hankel function of the first kind of order zero.
The scattered field $u^s(\cdot; d)$ is a radiating solution to the Helmholtz equation in $\R^2\ba\ov{B}$ such that the Green's formula holds \cite{CK}
\be\label{GreenFormula}
u^s(x;d) = \int_{\pa B}\left\{u^s(y;d)\frac{\pa\Phi(x,y)}{\pa\nu(y)}-\frac{\pa u^s(y;d)}{\pa\nu(y)}\Phi(x,y)\right\}ds(y), \quad x\in \R^2\ba\ov{B}.
\en
Letting $x$ tend to the boundary $\pa B$ and using the jump relations, it can be shown that
\[
(\phi,\psi):=\Big(u^s, \frac{\pa u^s}{\pa\nu}\Big)\Big|_{\pa B}\in H^{1/2}(\pa B)\times H^{-1/2}(\pa B)
\]
solves the following boundary integral equations
\be
\label{phipsi1} \phi(x)=2\int_{\pa B}\left\{\phi(y)\frac{\pa\Phi(x,y)}{\pa\nu(y)}-\psi(y)\Phi(x,y)\right\}ds(y), \quad x\in \pa B,\\
\label{phipsi2} \psi(x)=2\frac{\pa}{\pa\nu(x)}\int_{\pa B}\left\{\phi(y)\frac{\pa\Phi(x,y)}{\pa\nu(y)}-\psi(y)\Phi(x,y)\right\}ds(y), \quad x\in \pa B.
\en
For later use, we define the space
\ben
W:=\{(\phi,\psi)\in H^{1/2}(\pa B)\times H^{-1/2}(\pa B): (\phi,\psi)\,\mbox{ is a solution to}\, \eqref{phipsi1}-\eqref{phipsi2}.\}.
\enn

Using the Green's formula \eqref{GreenFormula}, $u^{\infty}(\cdot;d)$ has the following form (cf. \cite{KirschGrinberg})
\be\label{scatteringamplitude}
u^{\infty}(\hat{x};d)=\int_{\pa B}\left\{u^s(y;d)\frac{\pa e^{-ik\hat{x}\cdot y}}{\pa\nu(y)}-\frac{\pa u^s}{\pa\nu}(y; d)e^{-ik\hat{x}\cdot y}\right\}ds(y),\quad \hat{x}\in S^{1}.
\en
Note that the Cauchy data $\Big(u^s, \frac{\pa u^s}{\pa\nu}\Big)\Big|_{\pa B}$ is independent of the variable $\hat{x}$.
From \eqref{scatteringamplitude} the far field pattern can be computed in any direction if the Cauchy data $\Big(u^s, \frac{\pa u^s}{\pa\nu}\Big)\Big|_{\pa B}$ is known.

Let $S_0^{1}$ be the measurement surface, which is an open subset of the unit sphere $S^{1}$ with nonempty interior (open relative to $S^{1}$).
If we already know the far field pattern in $S_0^{1}$,
then it is natural to approximate the Cauchy data $\Big(u^s, \frac{\pa u^s}{\pa\nu}\Big)\Big|_{\pa B}$ by solving the following integral equation
\be\label{Cauchysolveequation}
F(\phi(\cdot;d),\psi(\cdot;d))(\hat{x})=u^{\infty}(\hat{x};d),\quad \hat{x}\in \,S_0^{1},
\en
where $F: W\rightarrow L^2(S_0^{1})$ is defined by
\be\label{F}
F(\phi(\cdot;d),\psi(\cdot;d))(\hat{x}):= \int_{\pa B}\left\{\phi(y;d)\frac{\pa e^{-ik\hat{x}\cdot y}}{\pa\nu(y)}-\psi(y;d)e^{-ik\hat{x}\cdot y}\right\}ds(y),\quad \hat{x}\in S_0^{1}.
\en
%We want to remark that the right hand side of \eqref{F} is actually the far field pattern of the scattered field with Cauchy data $(\phi,\psi)$ on $\pa B$.
\begin{theorem}\label{Fproperty}
The operator $F: W\rightarrow L^2(S_0^{1})$ is compact, injective with dense range in $L^2(S_0^{1})$.
\end{theorem}
\begin{proof}
The operator $F$ is certainly compact since its kernel is analytic in both variables.

Let $(\phi,\psi)\in W$ satisfy $F(\phi(\cdot;d),\psi(\cdot;d))(\hat{x})=0$ in $S_0^{1}$.
By analyticity, we have
\be\label{farfield0}
\int_{\pa B}\left\{\phi(y;d)\frac{\pa e^{-ik\hat{x}\cdot y}}{\pa\nu(y)}-\psi(y;d)e^{-ik\hat{x}\cdot y}\right\}ds(y)=0,\quad \hat{x}\in S^{1}.
\en
Note that the left hand side of \eqref{farfield0} is actually the far field pattern of the scattered field $w^s$ given by
\ben
w^s(x):=\int_{\pa B}\left\{\phi(y;d)\frac{\pa\Phi(x,y)}{\pa\nu(y)}-\psi(y;d)\Phi(x,y)\right\}ds(y), \quad x\in \R^2\ba\ov{B}.
\enn
By Rellich's lemma and \eqref{farfield0}, $w^s$ vanishes in $\R^2\ba\ov{B}$. Now, jump relations yield
\ben
0=\frac{1}{2}\phi(x)+\int_{\pa B}\left\{\phi(y)\frac{\pa\Phi(x,y)}{\pa\nu(y)}-\psi(y)\Phi(x,y)\right\}ds(y), \quad x\in \pa B,\\
0=\frac{1}{2}\psi(x)+\frac{\pa}{\pa\nu(x)}\int_{\pa B}\left\{\phi(y)\frac{\pa\Phi(x,y)}{\pa\nu(y)}-\psi(y)\Phi(x,y)\right\}ds(y), \quad x\in \pa B.
\enn
Recall that $(\phi,\psi)\in W$, which implies that $(\phi,\psi)$ is also a solution of \eqref{phipsi1}-\eqref{phipsi2}.
Hence, $\phi=\psi=0$ and $F$ is injective.

We consider the adjoint $F^{\ast}$ of $F$ and show that it is injective as well which proves the denseness of the range of $F$.
For all $h\in L^2(S_0^{1})$ we extend $h$ by zero in $S^{1}\ba S_0^{1}$ to obtain $h\in L^2(S^{1})$. Recall the Herglotz wave function $v_h$ of the form
\ben
v_h(y):=\int_{S^{1}} e^{iky\cdot\hat{x}}h(\hat{x})ds(\hat{x}),\quad y\in \R^2.
\enn
Then we obtain that the adjoint operator $F^{\ast}: L^2(S_0^{1})\rightarrow H^{-1/2}(\pa B)\times H^{1/2}(\pa B)$ is given by
\be\label{Fast}
F^{\ast}h=\left(\frac{\pa v_h}{\pa\nu}, \;\;-v_h\right)\Big|_{\pa B}.
\en
Interchanging the order of integration, we have
\ben
(F(\phi,\psi), h)_{L^2(S_0^{1})}
&=&\int_{S_0^{1}}\int_{\pa B}\left\{\phi(y;d)\frac{\pa e^{-ik\hat{x}\cdot y}}{\pa\nu(y)}-\psi(y;d)e^{-ik\hat{x}\cdot y}\right\}ds(y)\ov{h(\hat{x})}ds(\hat{x})\cr
&=&\int_{\pa B}\left\{\phi(y;d)\frac{\pa \ov{v_h(y)}}{\pa\nu(y)}-\psi(y;d)\ov{v_h(y)}\right\}ds(y)\cr
&=&\Big\langle\Big(\phi,\psi\Big), \Big(\frac{\pa v_h}{\pa\nu},-v_h\Big)\Big\rangle\cr
&=&\langle(\phi,\psi), F^{\ast}h\rangle,
\enn
where the last two equalities hold in the sense of dual paring $\langle H^{1/2}(\pa B)\times H^{-1/2}(\pa B), H^{-1/2}(\pa B)\times H^{1/2}(\pa B)\rangle$.
Here and in the following, $\ov{z}$ denotes the complex conjugate of $z\in\C$.

We proceed by showing that the adjoint operator $F^{\ast}$ is injective. Let $h\in L^2(S_0^{1})$ be such that $F^{\ast}h=0$ on $\pa B$.
Again extending $h$ by zero in $S^{1}\ba S_0^{1}$ to obtain $h\in L^2(S^{1})$. We find that the Cauchy data of the Herglotz wave function $v_h$ vanishes on $\pa B$. Note that
the Herglotz wave function $v_h$ is an entire solution of the Helmhlotz equation in $\R^2$. Thus by Holmgren's uniqueness theorem we deduce that $v_h$ vanishes identically in $\R^2$.
This further implies that $h=0$ on $\pa B$ \cite{CK} and the proof is complete.
\end{proof}

{\bf Method of Green's Formula (MGF)}
\begin{itemize}
\item Given partial data $\mathbb{F}^{(l)}_{limit}$, set $\widetilde{\mathbb{F}}^{(l)}_{limit}$ by Theorem~\ref{blocksymmetric}.
\item Solve $(\phi,\psi)$ for \eqref{Cauchysolveequation} using $\mathbb{F}^{(l)}_{limit} \cup \widetilde{\mathbb{F}}^{(l)}_{limit}$ by Tikhonov regularization.
\item Use \eqref{scatteringamplitude} to obtain $\mathbb{F}_{full}$.
\end{itemize}
%{\em As described at the begin of this section, assume that we have $2m$ equidistantly distributed incident directions in $S^1$, and we choose
%the first $l$ directions as the observation directions in $S_0^{1}$. Then, practically, we obtain the far field pattern as shown in the matrix $\mathbb{F}^{(l)}_{limit}$.
%Theorem \ref{Fproperty} indicated that, for every $d_i\in S^1$, we may find a pair of approximate solution $(\phi,\psi)$ of the equations \eqref{phipsi1}-\eqref{phipsi2} and
%\eqref{Cauchysolveequation} by using the Tikhonov regularization.
%Inserting this into formula \eqref{scatteringamplitude}, we then obtain the approximate far field pattern in other directions in $S^{1}\ba S_0^{1}$.}

\subsection{Method of Single Layer Potential}\label{sec2.3}

We consider the scattered field $u^s$ in the form of a single-layer potential
\ben
u^s(x)=\int_{\pa B}\Phi(x,y)\phi(y)ds(y), \quad x\in\R^2\ba\ov{B}
\enn
with an unknown density $\phi\in L^2(\pa B)$. Its far field pattern is given by
\be\label{farfield3}
u^{\infty}(\hat{x})=\int_{\pa B}e^{ik\hat{x}\cdot y}\phi(y)ds(y), \quad \hat{x}\in\, S^1.
\en
Inspired by this, we introduce the following integral equation of the first kind
\be\label{Sinfty}
S_{\infty}\phi=u^{\infty},
\en
where the far field integral operator $S_{\infty}: L^2(\pa B)\rightarrow L^2(S^1)$ is defined by
\be\label{Sinfty2}
(S_{\infty}\phi)(\hat{x}):=\int_{\pa B}e^{-ik\hat{x}\cdot y}\phi(y)ds(y).
\en
The properties of the far field integral operator $S_{\infty}$ have been collected in the following theorem.
\begin{theorem}
The far field integral operator $S_{\infty}: L^2(\pa B)\rightarrow L^2(S^1)$ is injective and has dense range provided $k^2$ is not a Dirichlet eigenvalue for the negative Laplacian
in $B$.
\end{theorem}
We omit the proof since, except for minor adjustments, they literally coincide with those of Theorem 5.19 in \cite{CK}, where, the 3D case is considered.
The requirement on $k^2$ is not essential since we have the freedom to choose $B$.

To recover {\em full-aperture} data $\mathbb{F}_{full}$,
one may firstly solve the equation \eqref{Sinfty} by the Tikhonov regularization in $L^2(S^1_0)$ and then insert the solution $\phi$
into $u^{\infty}(\hat{x}):=(S_{\infty}\phi)(\hat{x})$ to obtain the missed data. Again, the integral operator $S_{\infty}$ has an analytic kernel and therefore equation
\eqref{Sinfty} is severely ill-posed.\\

{\bf Method of Single Layer Potential (MSLP)}
\begin{itemize}
\item Given partial data $\mathbb{F}^{(l)}_{limit}$, set $\widetilde{\mathbb{F}}^{(l)}_{limit}$ by Theorem~\ref{blocksymmetric}.
\item Solve $\phi$ for \eqref{Sinfty} using $\mathbb{F}^{(l)}_{limit} \cup \widetilde{\mathbb{F}}^{(l)}_{limit}$ by Tikhonov regularization.
\item Use  \eqref{farfield3} to obtain $\mathbb{F}_{full}$.
\end{itemize}

%{\em For every $d_i\in S^1$, we firstly look for an approximate solution $\phi$ of the equation \eqref{Sinfty} by using the Tikhonov regularization.
%Then, inserting the approximate solution $\phi$ into formula \eqref{farfield3}, the approximate far field pattern in other directions in $S^{1}\ba S_0^{1}$ can be obtained.}

\subsection{From limited-aperture to full-aperture}
The direct application of the above two methods does not produce good recovery of the full-aperture data due to the severe ill-posedness.
To improve the result, we propose a step by step alternative technique which makes use of the symmetry of $\mathbb{F}_{full}$.
Roughly speaking, we use {\bf MGF} or {\bf MSLP} to recover a few data using the known data. Then Theorem~\ref{blocksymmetric} is used to obtain more data.
The process is repeated until $\mathbb{F}_{full}$ is recovered.

%Using some idea from the least square method and the iterative method, we have the following scheme to compute the {\em full-aperture} data.
Now we ready to introduce the algorithm to recover the full-aperture MSR.

\begin{itemize}
  \item[] {\bf DR-MSR:}
  \item Step 1. Measure the {\em limited-aperture} far field pattern $\mathbb{F}^{(l)}_{limit}$.
  \item Step 2. Using {\bf MGF} or {\bf MSLP}, recover the far field pattern
                \ben
                M_{new}:=\{\hat{x}_{l+1}, \hat{x}_{l+1},\cdots \hat{x}_{l+t}, \hat{x}_{2m-st-t+1}, \hat{x}_{2m-st-t+2},\cdots,\hat{x}_{2m-st}\},
                \enn
                i.e., compute data in $2t$ new directions close to known data.
  \item Step 3. Recover $\widetilde{\mathbb{F}}^{(l)}_{limit}$ in \eqref{MSR-l} using Theorem~\ref{blocksymmetric}.
  \item Step 4. If $\mathbb{F}_{full}$ is obtained, stop. Otherwise, set $l=l+t, s=s+1$ and go to Step 2.
\end{itemize}

\begin{remark}
The scheme makes no use of the boundary conditions or topological properties of the underlying object $\Om$. In other words, the full-aperture data is
retrieved without any a priori information on $\Om$.
\end{remark}

\section{Numerical Examples and Discussions}
\label{NEs}
\setcounter{equation}{0}
The numerical examples are divided into two groups.
We first present some numerical examples to demonstrate how to use the proposed methods to recover data.
The second group of numerical examples are to test the recovered data, which are used by some non-iterative methods to reconstruct the support of the scatterer.
The results show that the reconstruction improves significantly, indicating that the recovered data does help in certain cases.
%All the programs in our experiments are written in Matlab.
%
%There are totally three groups of numerical tests to be considered, and they are
%respectively referred to as {\bf BlockSymmetric, DataRetrieval} and {\bf SamplingMethods}.
Two scatterers are considered (see Fig.~\ref{truedomains}):
\be
\label{kite}&\mbox{\rm Kite:}&\quad x(t)\ =(\cos t+0.65\cos 2t-0.65, 1.5\sin t),\quad 0\leq t\leq2\pi,\\
\label{peanut}&\mbox{\rm Peanut:}&\quad x(t)\ =\sqrt{3\cos^2 t+1}(\cos t, \sin t),\quad 0\leq t\leq2\pi.
%\label{pear}&\mbox{\rm Pear:}&\quad x(t)\ =(2+0.3\cos 3t)\ (\cos t, \sin t),\quad 0\leq t\leq2\pi,\\
%\label{roundsquare}&\mbox{\rm RoundSquare:}&\quad x(t)\ =(\cos^3 t+\cos t, \sin^3 t+\sin t),\quad 0\leq t\leq2\pi,
\en
\begin{figure}[htbp]
  \centering
  \subfigure[\textbf{Kite}]{
    \includegraphics[width=2.6in]{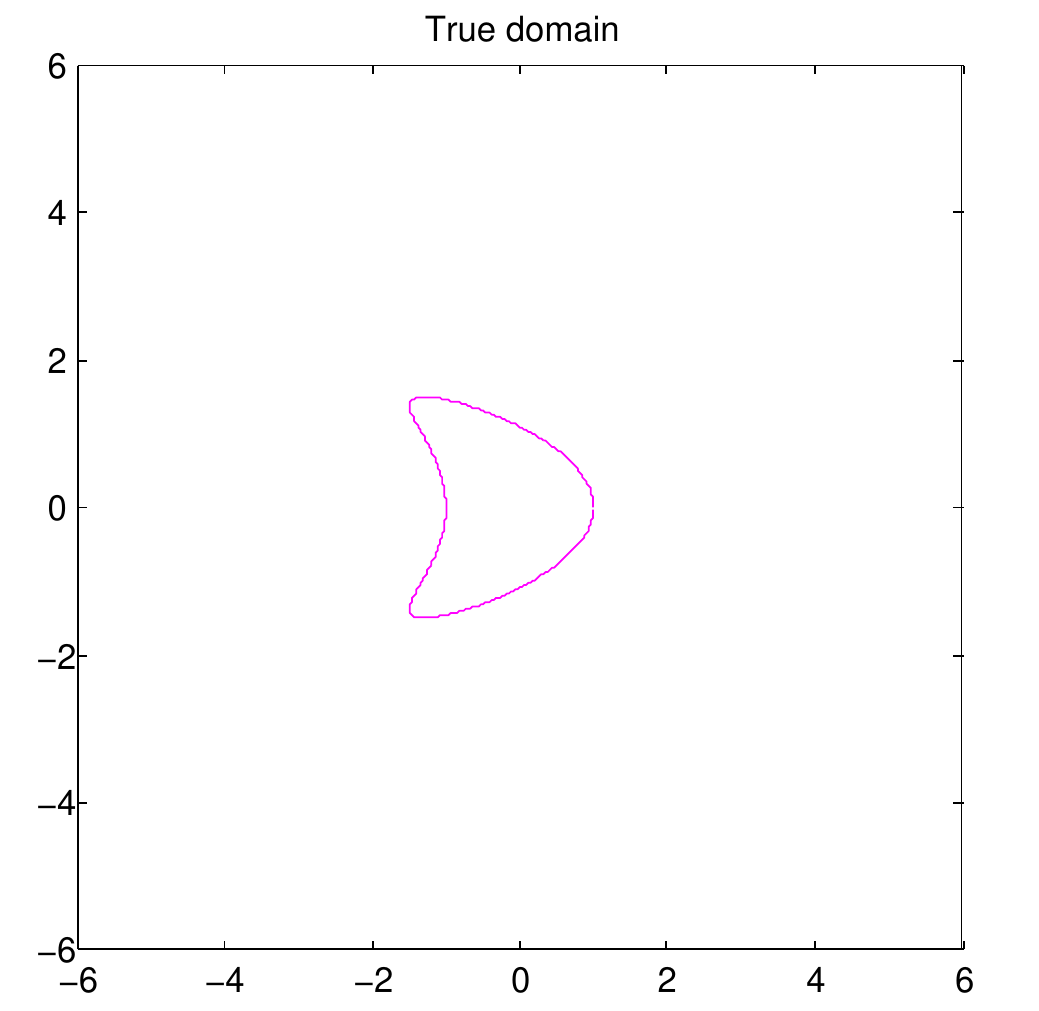}}
  \subfigure[\textbf{Peanut}]{
    \includegraphics[width=2.6in]{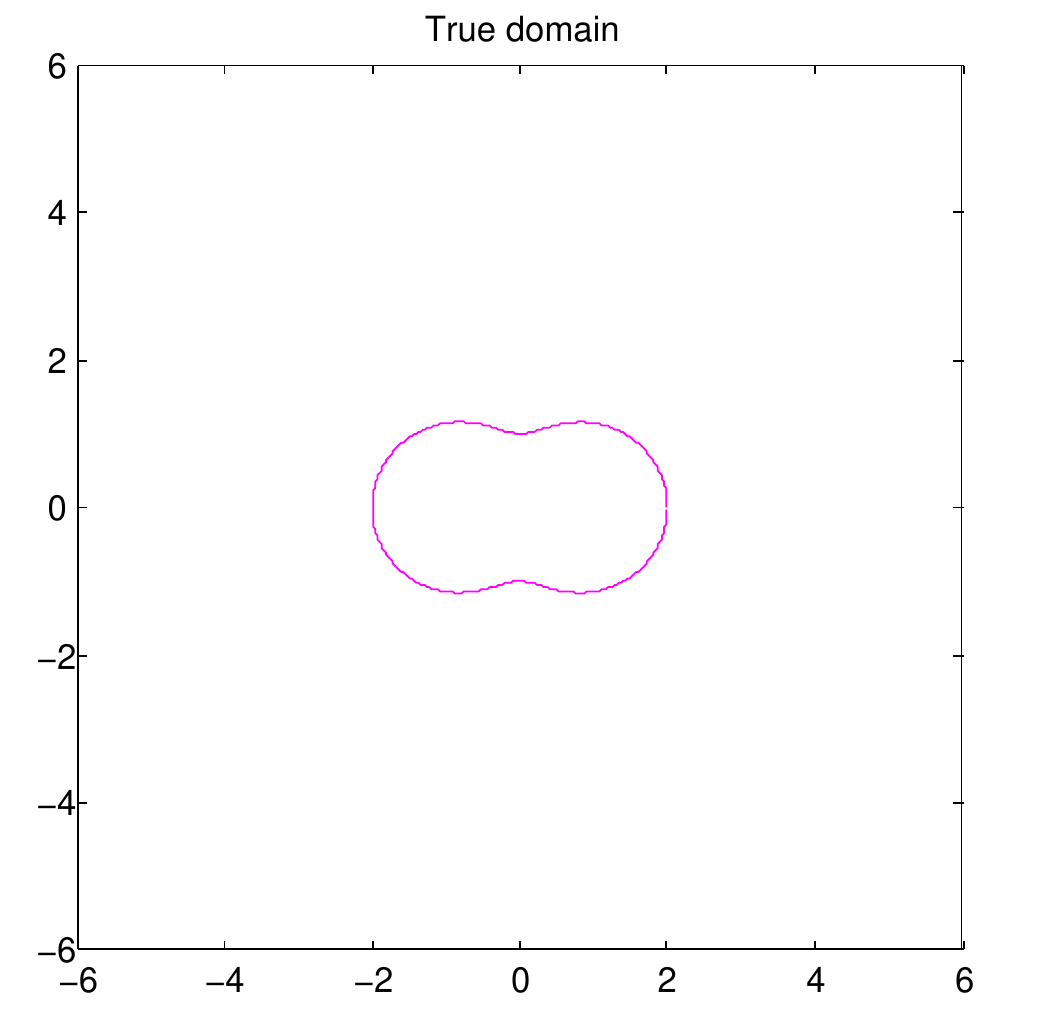}}
%  \subfigure[\textbf{Pear}]{
%    \includegraphics[width=1.4in]{pic/PearTrue.pdf}}
%  \subfigure[\textbf{RoundSquare}]{
%    \includegraphics[width=1.4in]{pic/RoundSquareTrue.pdf}}
\caption{Domains considered.}
\label{truedomains}
\end{figure}

The synthetic scattering data are generated by the boundary integral equation method, i.e.,
$u^\infty_{p,q},\,p,q=1,2,\cdots,2m$, for $2m$ equidistantly distributed directions in $(0, 2\pi]$.
We first verify Theorem \ref{blocksymmetric}.
Let $k=6$, $m=4$ and take the kite as an example.
The MSR matrix $\mathbb{F}_{full}$ is $8\times 8$. The four block matrices are given as follows.
\ben
\mathbb{F}_{11}
&=&\left(
   \begin{array}{cccc}
    -2.6282 + 1.8817i &  0.1698 + 0.4158i &  0.1657 - 0.2286i  & 1.0722 - 0.6313i \\
   0.0028 - 0.9694i & -2.5830 + 1.9160i &  0.3264 + 0.1581i & -0.4424 - 0.9227i \\
  -0.0740 + 0.7809i &  0.2839 - 0.5024i & -2.4052 + 1.5689i &  0.3264 + 0.1581i  \\
   0.1929 - 0.5886i &  0.2202 + 0.4858i  & 0.2839 - 0.5024i & -2.5830 + 1.9160i  \\
   \end{array}
 \right),\\
\mathbb{F}_{22}
&=&\left(
  \begin{array}{cccccccc}
   -2.6282 + 1.8817i &  0.0028 - 0.9694i & -0.0740 + 0.7809i  & 0.1929 - 0.5886i\\
   0.1698 + 0.4158i  &-2.5830 + 1.9160i  & 0.2839 - 0.5024i   &0.2202 + 0.4858i\\
   0.1657 - 0.2286i  & 0.3264 + 0.1581i  &-2.4052 + 1.5689i   &0.2839 - 0.5024i\\
   1.0722 - 0.6313i  &-0.4424 - 0.9227i  & 0.3264 + 0.1581i  &-2.5830 + 1.9160i  \\
  \end{array}
\right),\\
\mathbb{F}_{12}
&=&\left(
  \begin{array}{cccccccc}
   -0.5250 + 0.1132i &  1.0722 - 0.6313i &  0.1657 - 0.2286i &  0.1698 + 0.4158i \\
   1.0722 - 0.6313i & -0.0050 - 0.3054i  & 0.1510 + 0.3285i & -0.5603 - 0.0594i\\
   0.1657 - 0.2286i &  0.1510 + 0.3285i & -0.3128 - 0.5104i & -0.6526 - 1.3338i\\
   0.1698 + 0.4158i & -0.5603 - 0.0594i & -0.6526 - 1.3338i & -0.0441 - 0.9080i \\
  \end{array}
\right),\\
\mathbb{F}_{21}
&=&\left(
  \begin{array}{cccccccc}
   -0.4473 - 0.3633i  & 0.1929 - 0.5886i & -0.0740 + 0.7809i  & 0.0028 - 0.9694i\\
   0.1929 - 0.5886i  &-0.0441 - 0.9080i  &-0.6526 - 1.3338i & -0.5603 - 0.0594i\\
  -0.0740 + 0.7809i  &-0.6526 - 1.3338i  &-0.3128 - 0.5104i &  0.1510 + 0.3285i\\
   0.0028 - 0.9694i  &-0.5603 - 0.0594i  & 0.1510 + 0.3285i & -0.0050 - 0.3054i \\
  \end{array}
\right).
\enn
It is obvious that Theorem \ref{blocksymmetric} holds:
\ben
\mathbb{F}_{11}= \mathbb{F}_{22}^T,\quad \mathbb{F}_{12}=\mathbb{F}_{12}^T\quad\mbox{ and }\quad \mathbb{F}_{21}= \mathbb{F}_{21}^T.
\enn

In the rest of the section, we fix $k=6$ and divide $(0, 2\pi)$ uniformly into $300\,(2m)$ directions.
Assuming that the incident directions cover the full aperture, the observation directions only span
a subset of $(0, 2\pi)$. In particular, let $\hat{x}:=(\cos\phi, \sin\phi)$ with the observation angle $\phi$. Then we consider the measurements for three cases:
\ben
(1)\,\phi\in (0,\pi/2),\quad (2)\,\phi\in (0,2\pi/3),\quad\mbox{and}\quad (3)\,\phi\in (0,\pi).
\enn
Namely, we have the synthetic data $\mathbb{F}^{(l)}_{limit}$ for $l=1,\ldots, 75$,  $l=1,\ldots, 100$, and $l=1,\ldots, 150$, respectively.
Then
$\mathbb{F}^{(l)}_{limit}$ is perturbed by random noises
\ben
\mathbb{F}_{limit}^{(l), \delta}\, =\, \mathbb{F}_{limit}^{(l)} +\delta\|\mathbb{F}_{limit}^{(l)}\|\frac{R_1+R_2 i}{\|R_1+R_2 i\|},
\enn
where $R_1$ and $R_2$ are two matrixes containing pseudo-random values
drawn from a normal distribution with mean zero and standard deviation one. The
value of $\delta$ is the noise level, which is taken as $\delta = 0.05$.

%Let $\mathbb{F}_{limit}^{l,\delta}$ be the matrix from the first $l$ column of $\mathbb{F}_{full}^{\delta}$, and thus denote the {\em limited-aperture} data.
%We have used $4m$ equidistantly distributed nodes to discrete the boundary $\pa \Om$ of the scatterer.

\subsection{Data Recovery}
Examples in this subsection are to test the validity of data recover algorithms proposed in Section~\ref{DRT}.
Given $\mathbb{F}_{limit}^{(l), \delta}$, the goal is to recover $\mathbb{F}_{full}^{\delta}$, the {\em full-aperture} MSR.
We take the kite as the scatterer. The artificial domain $B$ is chosen to be a disc centered at the origin with radius $5$.

Figures \ref{dataretrievalkite4}-\ref{dataretrievalkite2} show the data reconstructions with different measurement apertures
$(0,\pi/2), (0,2\pi/3),$ and $(0,\pi),$ respectively.
We show the recovered data for the incident direction $d=(1,0)$, i.e., the first row of $\mathbb{F}_{full}^{\delta}$ in Figures \ref{dataretrievalkite4}$(a)$-\ref{dataretrievalkite2}$(a)$.
Only few data are well reconstructed. This is reasonable because both methods, {\bf MGF} and {\bf MSLP}, involve solving severely ill-posed integral equations.
In particular, the symmetry does not apply in this case.
Similar results with respective to $d=(0,1)$ are shown in Figures \ref{dataretrievalkite4}$(b)$-\ref{dataretrievalkite2}$(b)$.

%Considering the severe ill-posedness of data retrieval for analytic functions \cite{Atkinson} and $5\%$ noise,
%the reconstructed data can be regarded to be a good approximation of the {\em full-aperture} measurements.

For incident angles in $[\pi, 3\pi/2]$, we have obtained nearly exact far field pattern for all observation directions. In Figures
\ref{dataretrievalkite4}$(c)(d)$-\ref{dataretrievalkite2}$(c)(d)$ we show results for two incident directions  $d=(-1,0)$ and $d=(0,-1)$.
This further verify the symmetric structure of the multi-static response matrix.\\

\begin{figure}[htbp!]
  \centering
  \subfigure[\textbf{$d=(1,0)$}]{
    \includegraphics[width=2.5in]{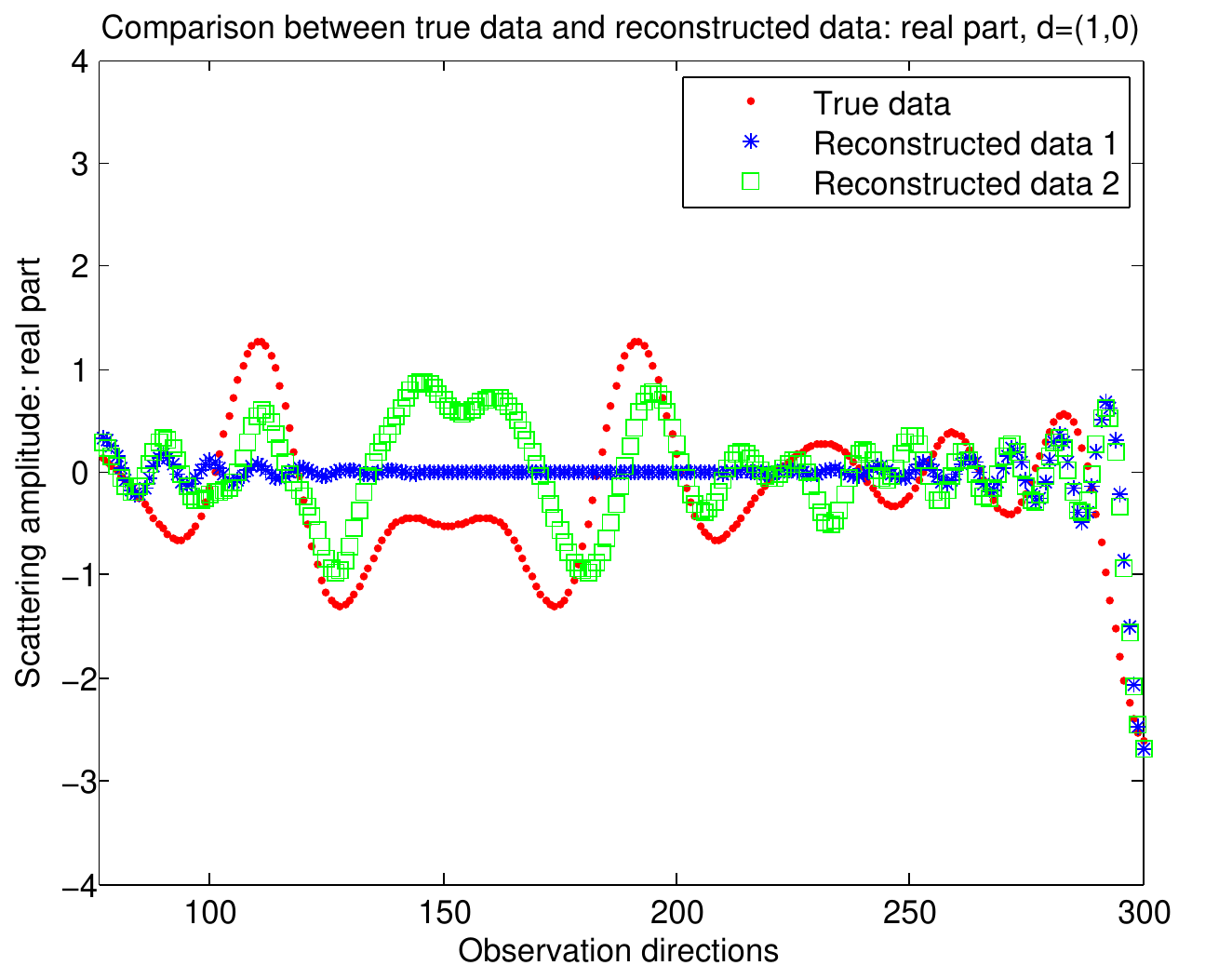}}
  \subfigure[\textbf{$d=(0,1)$}]{
    \includegraphics[width=2.5in]{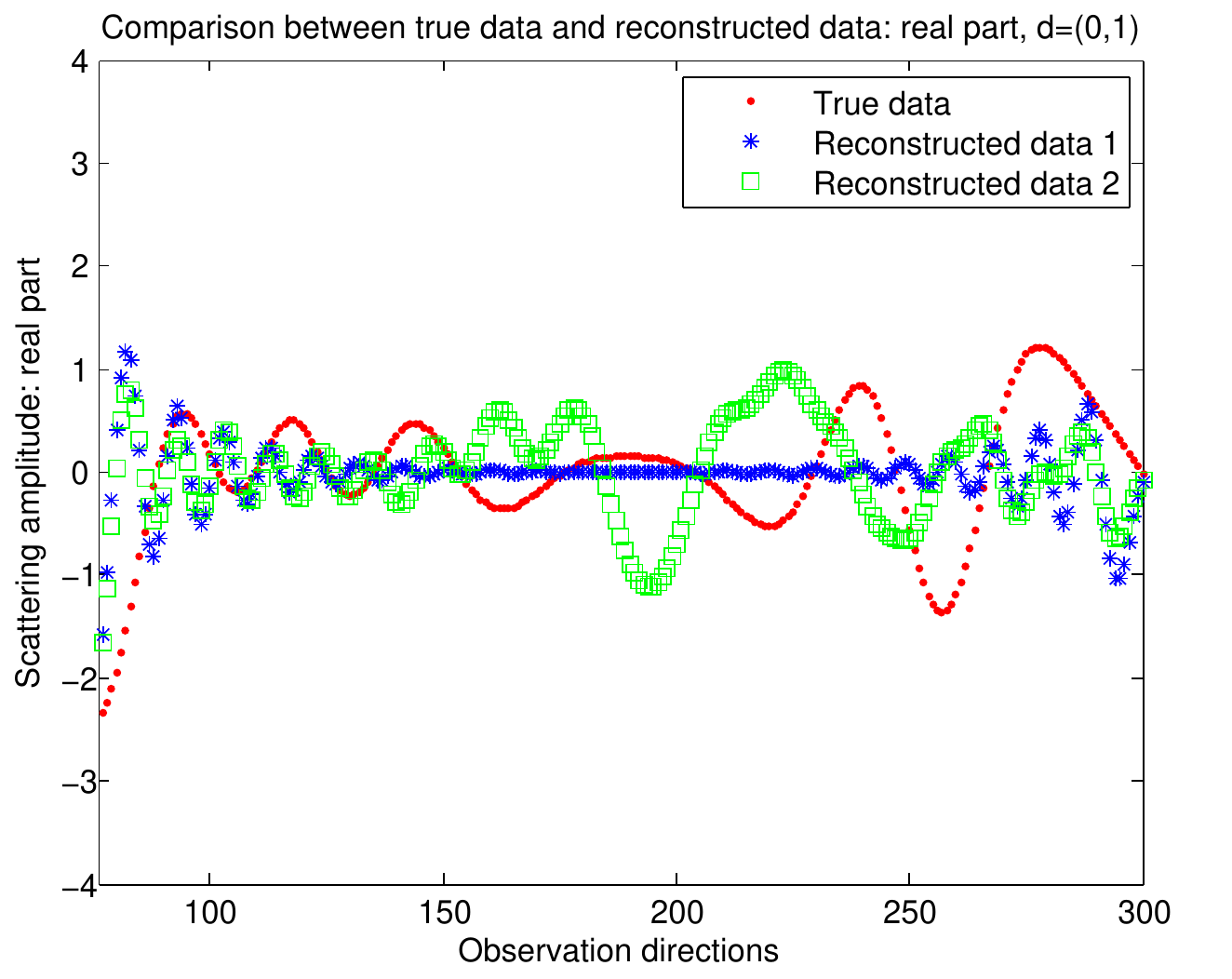}}
  \subfigure[\textbf{$d=(-1,0)$}]{
    \includegraphics[width=2.5in]{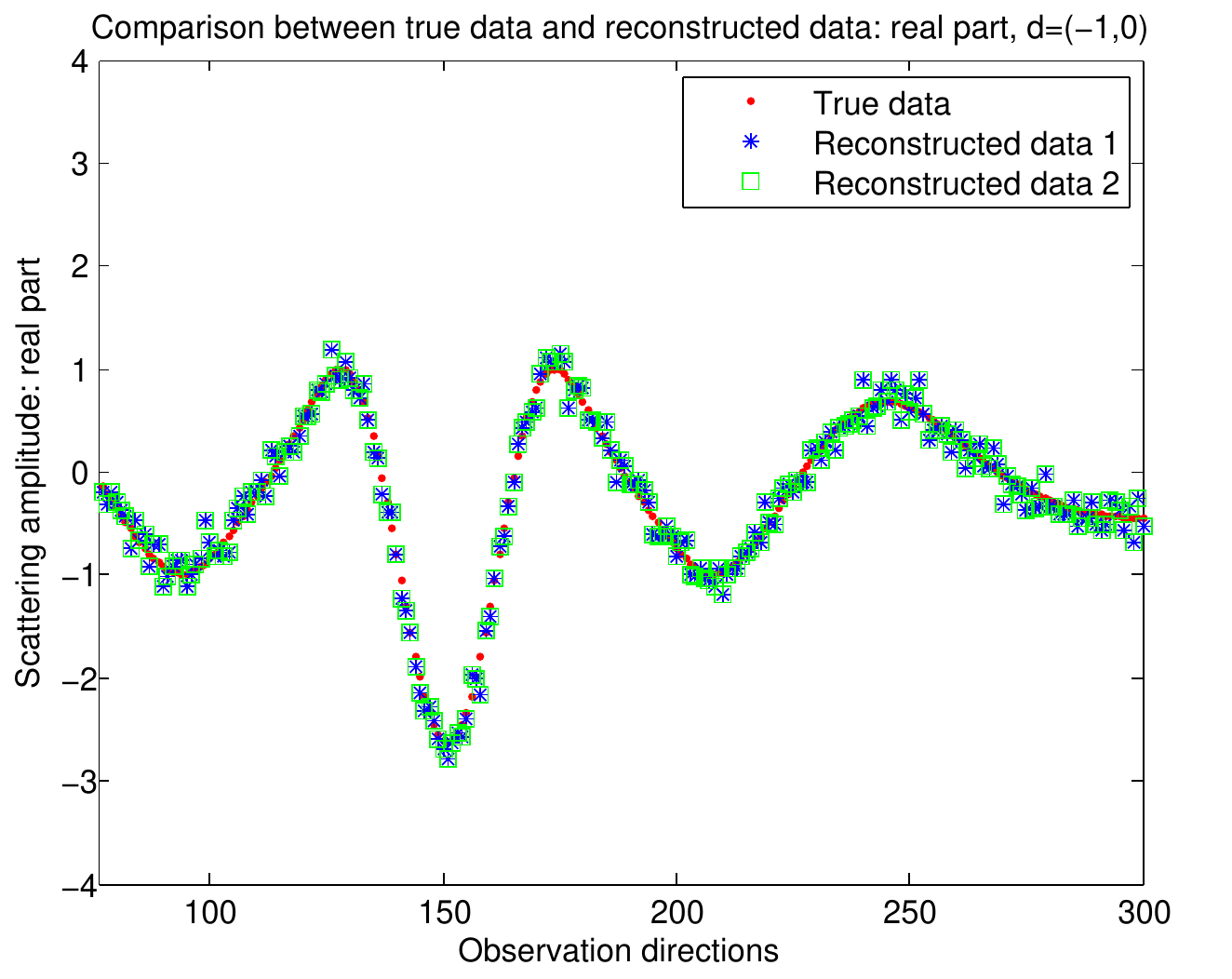}}
  \subfigure[\textbf{$d=(0,-1)$}]{
    \includegraphics[width=2.5in]{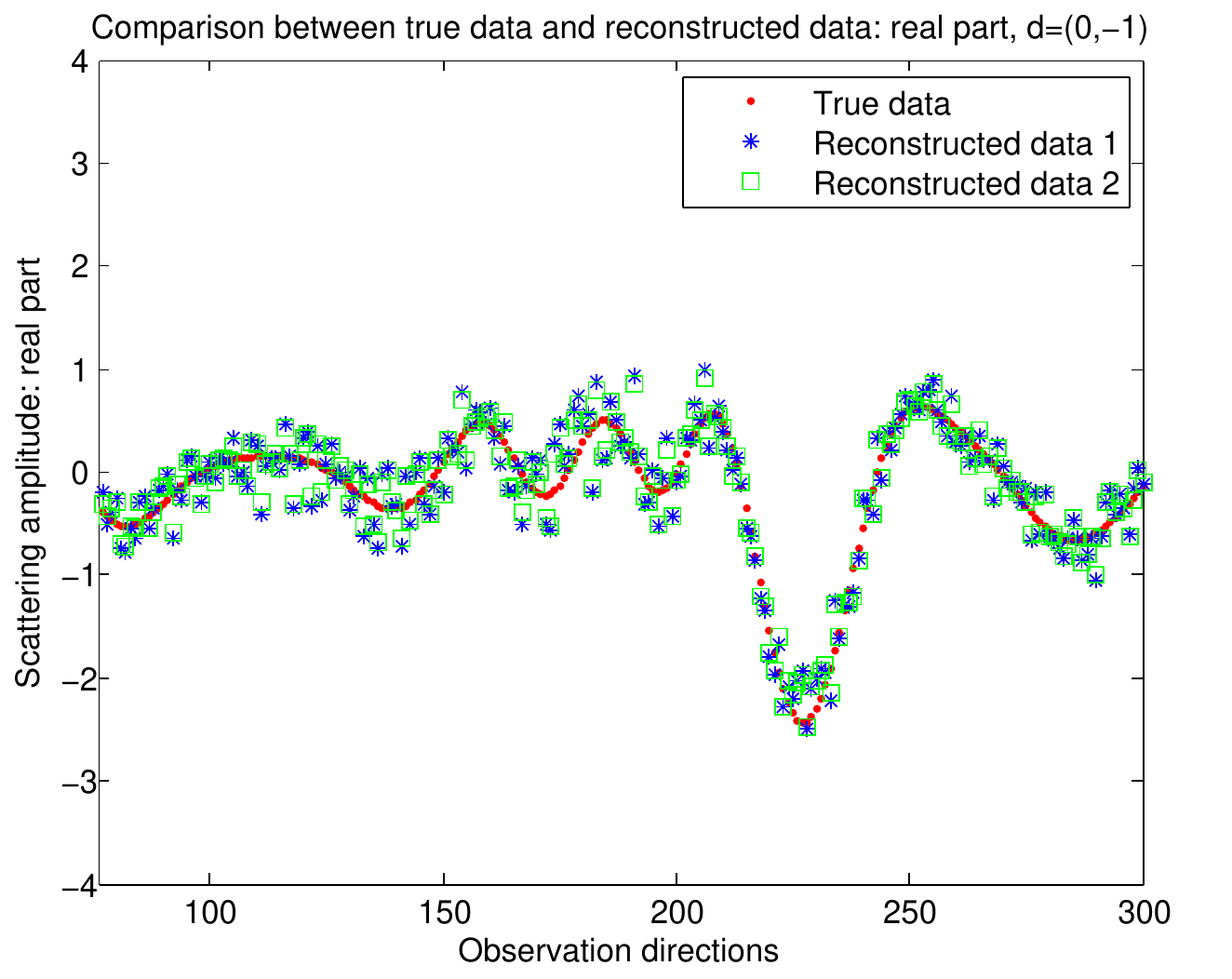}}
\caption{Exact data and recovered data with four different incident directions.
Reconstructed data $1$ is obtained by {\bf MGF}.
Reconstructed data $2$ is obtained by {\bf MSLP}.
The measurements are taken with observation angles $\phi\in (0,\pi/2)$. }
\label{dataretrievalkite4}
\end{figure}

\begin{figure}[htbp!]
  \centering
  \subfigure[\textbf{$d=(1,0)$}]{
    \includegraphics[width=2.5in]{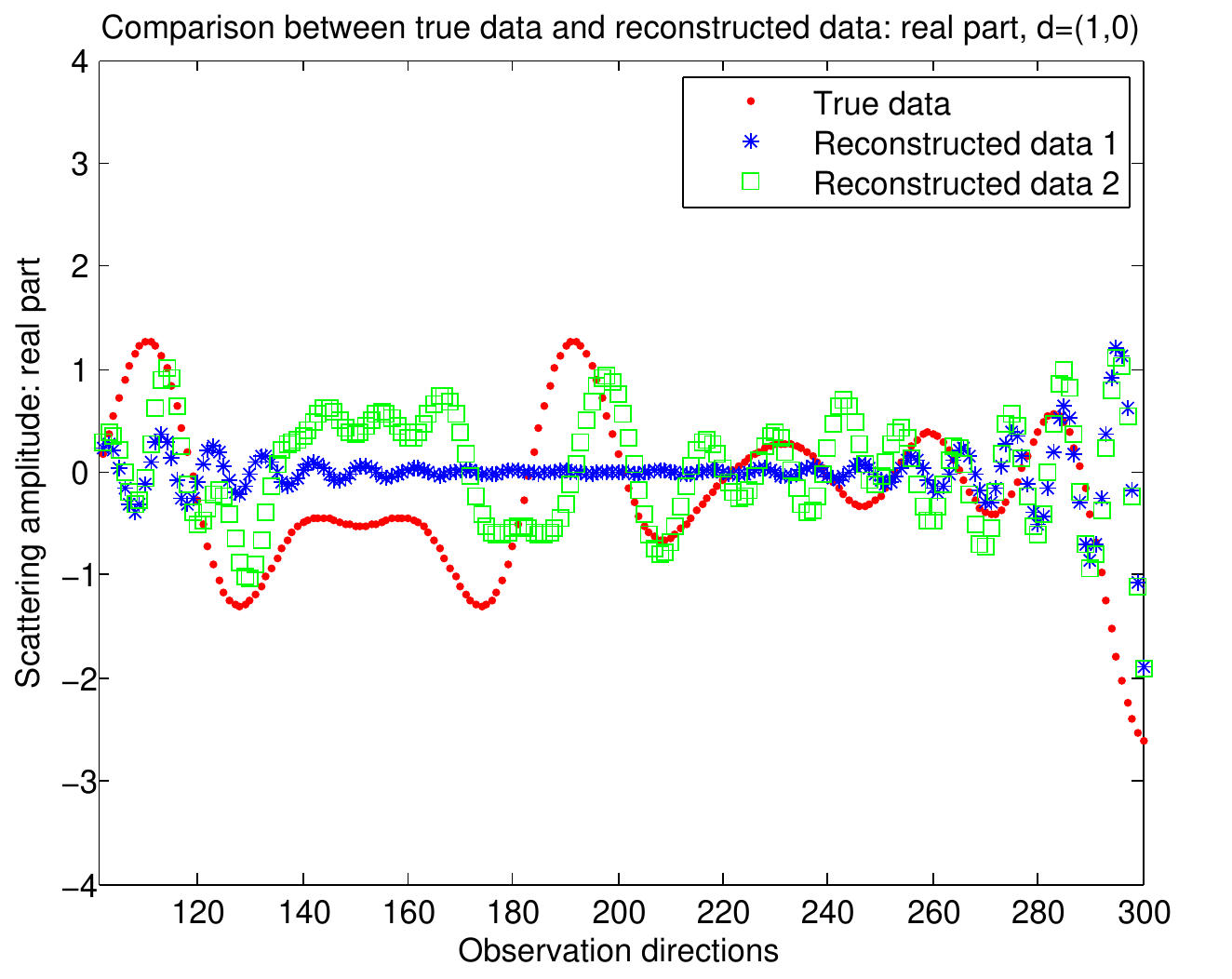}}
  \subfigure[\textbf{$d=(0,1)$}]{
    \includegraphics[width=2.5in]{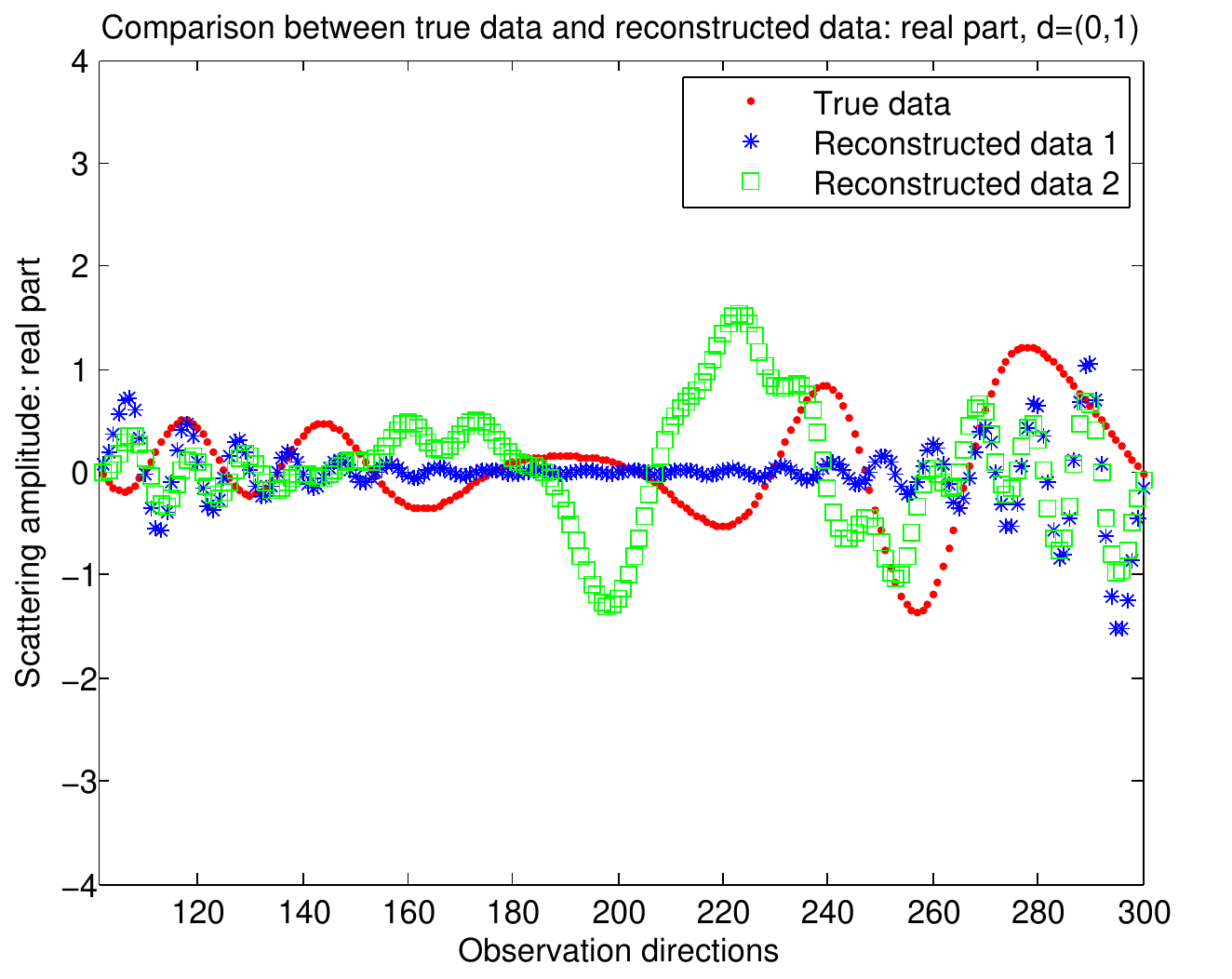}}
  \subfigure[\textbf{$d=(-1,0)$}]{
    \includegraphics[width=2.5in]{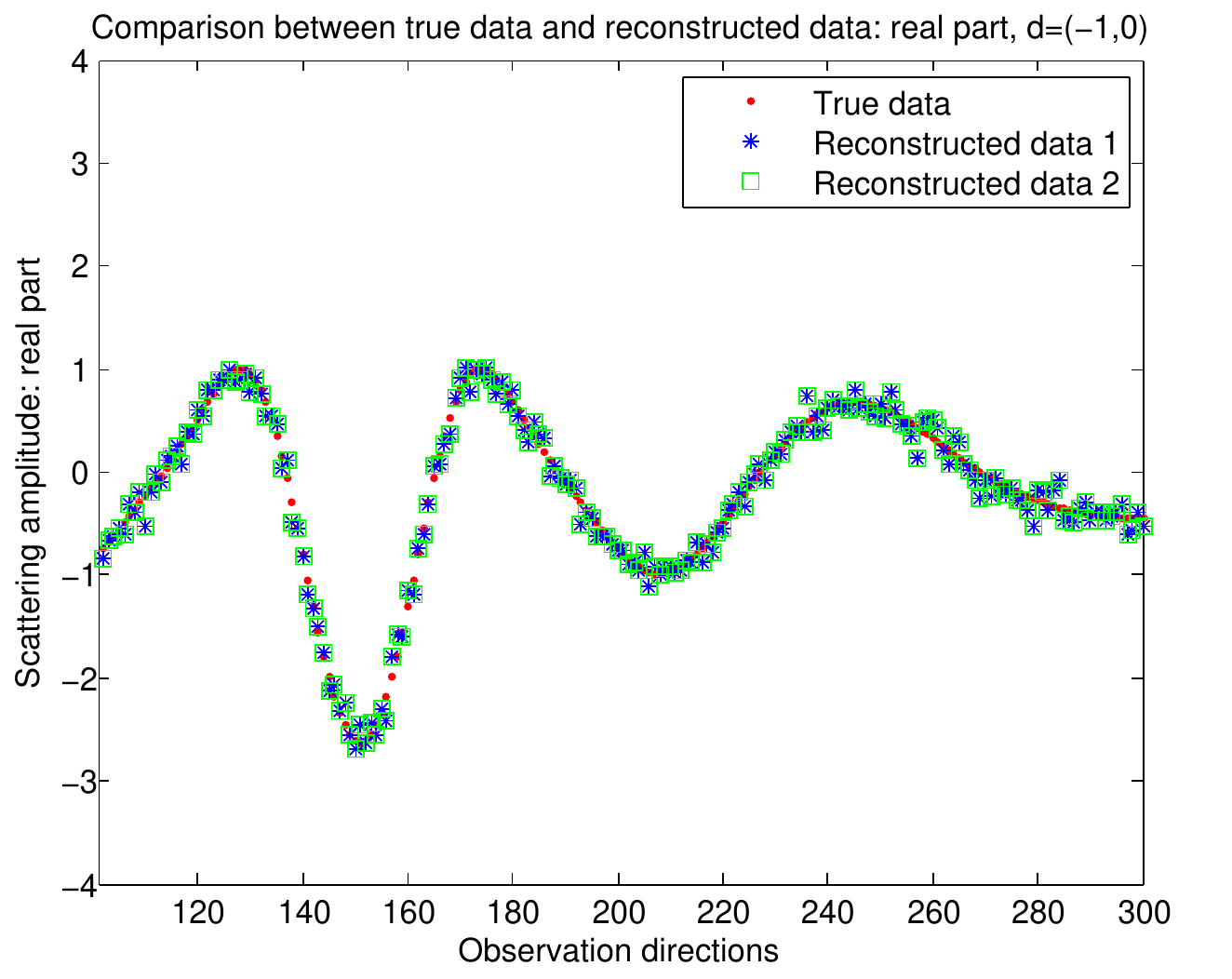}}
  \subfigure[\textbf{$d=(0,-1)$}]{
    \includegraphics[width=2.5in]{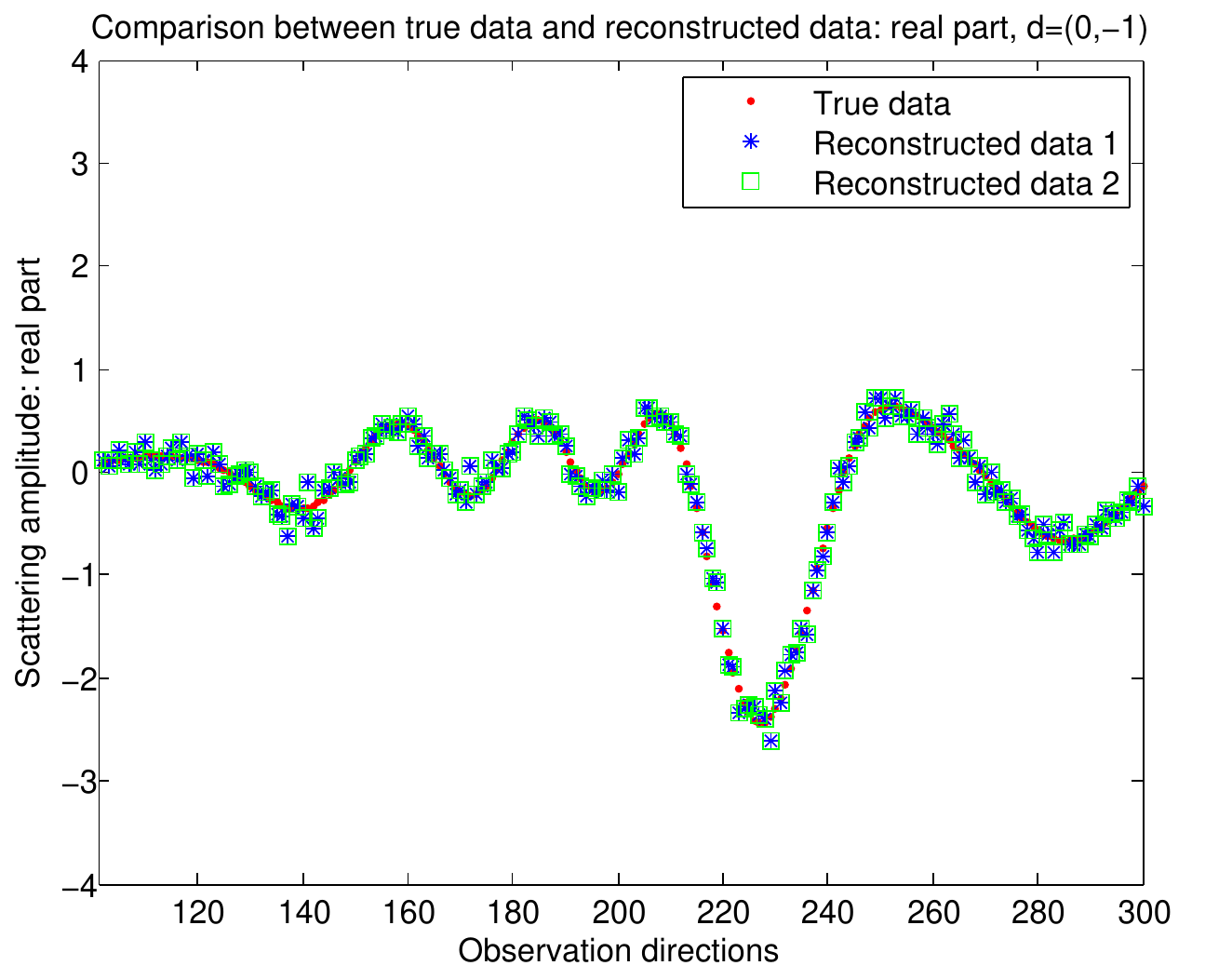}}
\caption{Exact data and recovered data with four different incident directions.
Reconstructed data $1$ is obtained by {\bf MGF}.
Reconstructed data $2$ is obtained by {\bf MSLP}.
The measurements are taken with observation angles  $\phi\in (0,2\pi/3)$. }
\label{dataretrievalkite3}
\end{figure}

\begin{figure}[htbp!]
  \centering
  \subfigure[\textbf{$d=(1,0)$}]{
    \includegraphics[width=2.5in]{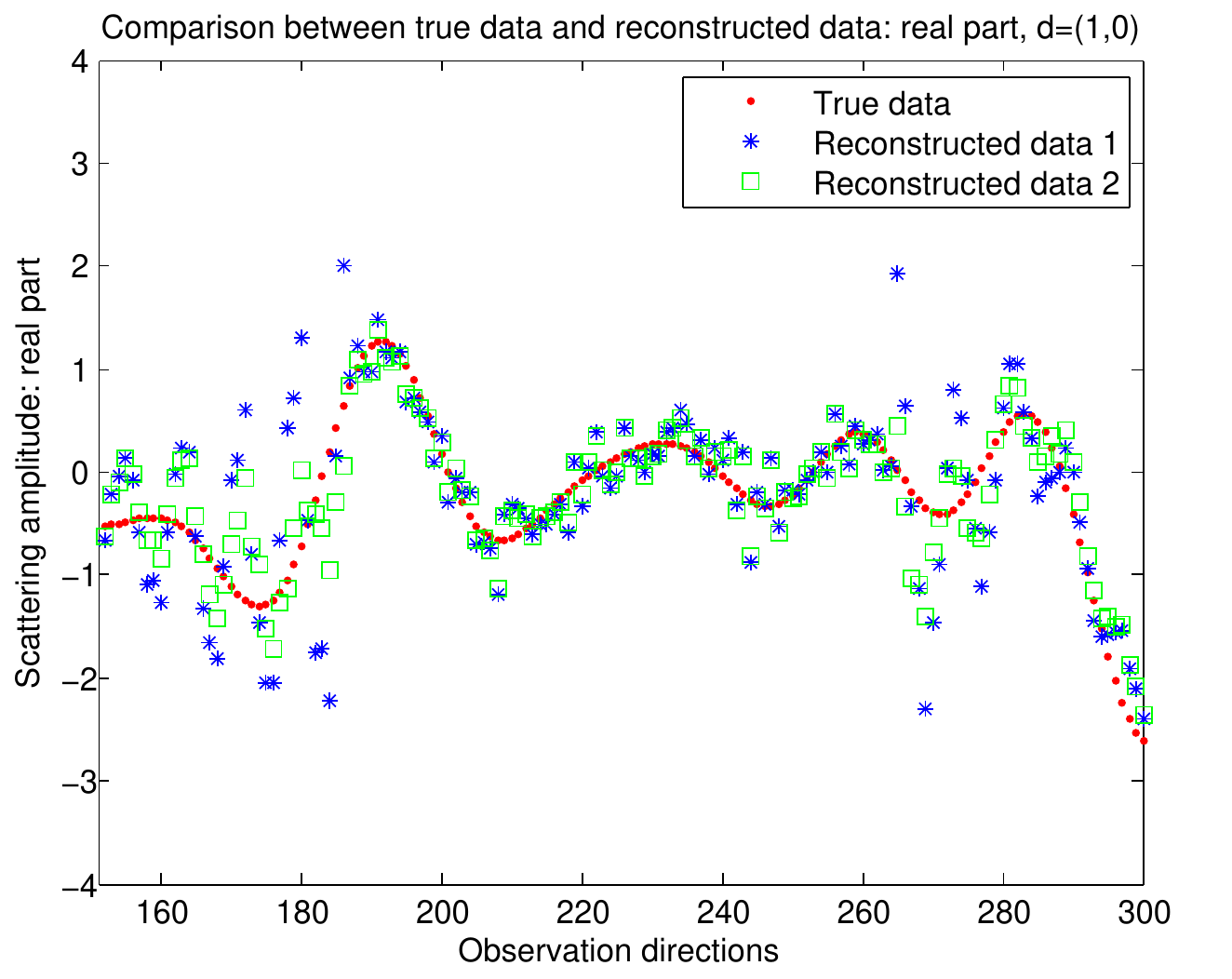}}
  \subfigure[\textbf{$d=(0,1)$}]{
    \includegraphics[width=2.5in]{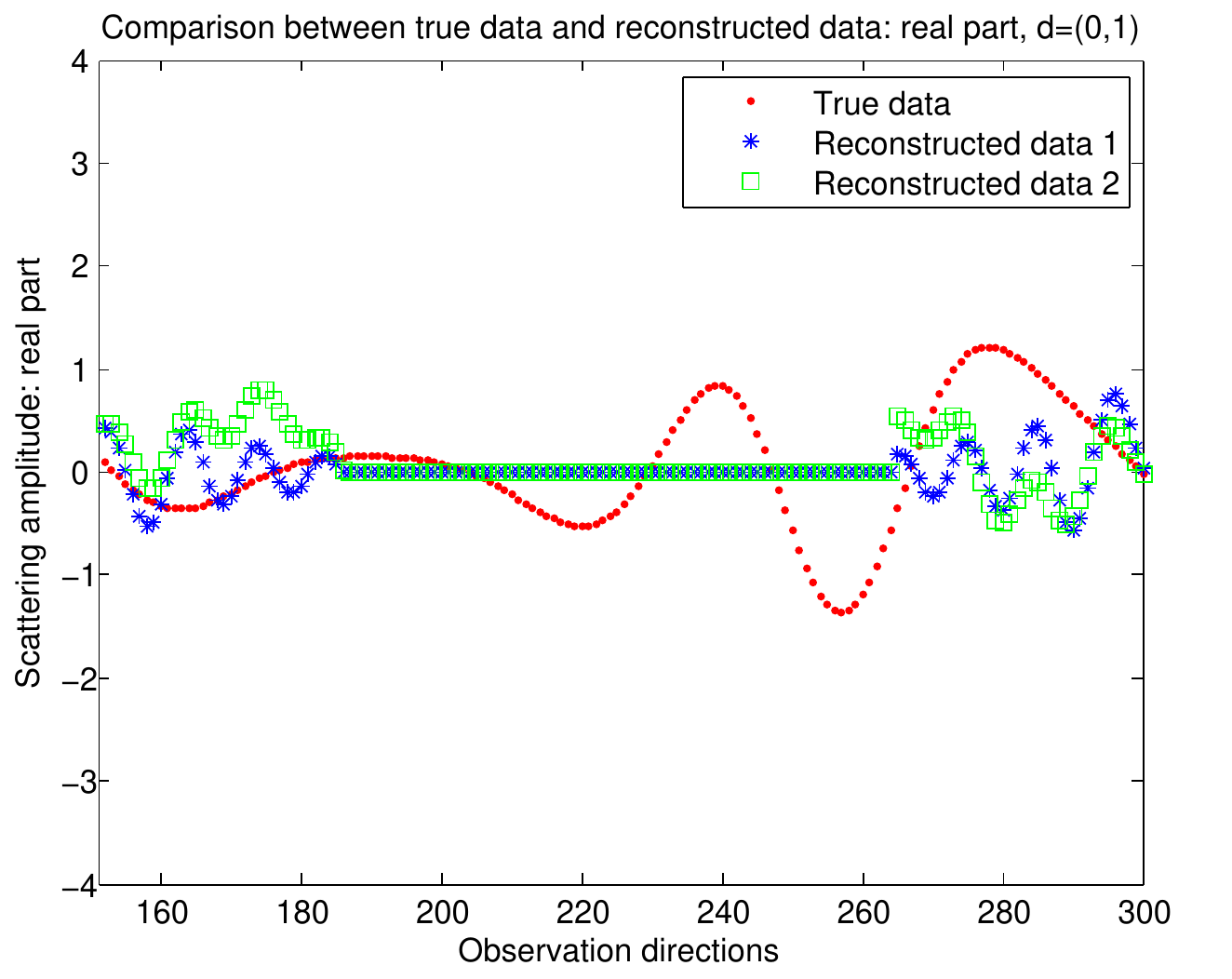}}
  \subfigure[\textbf{$d=(-1,0)$}]{
    \includegraphics[width=2.5in]{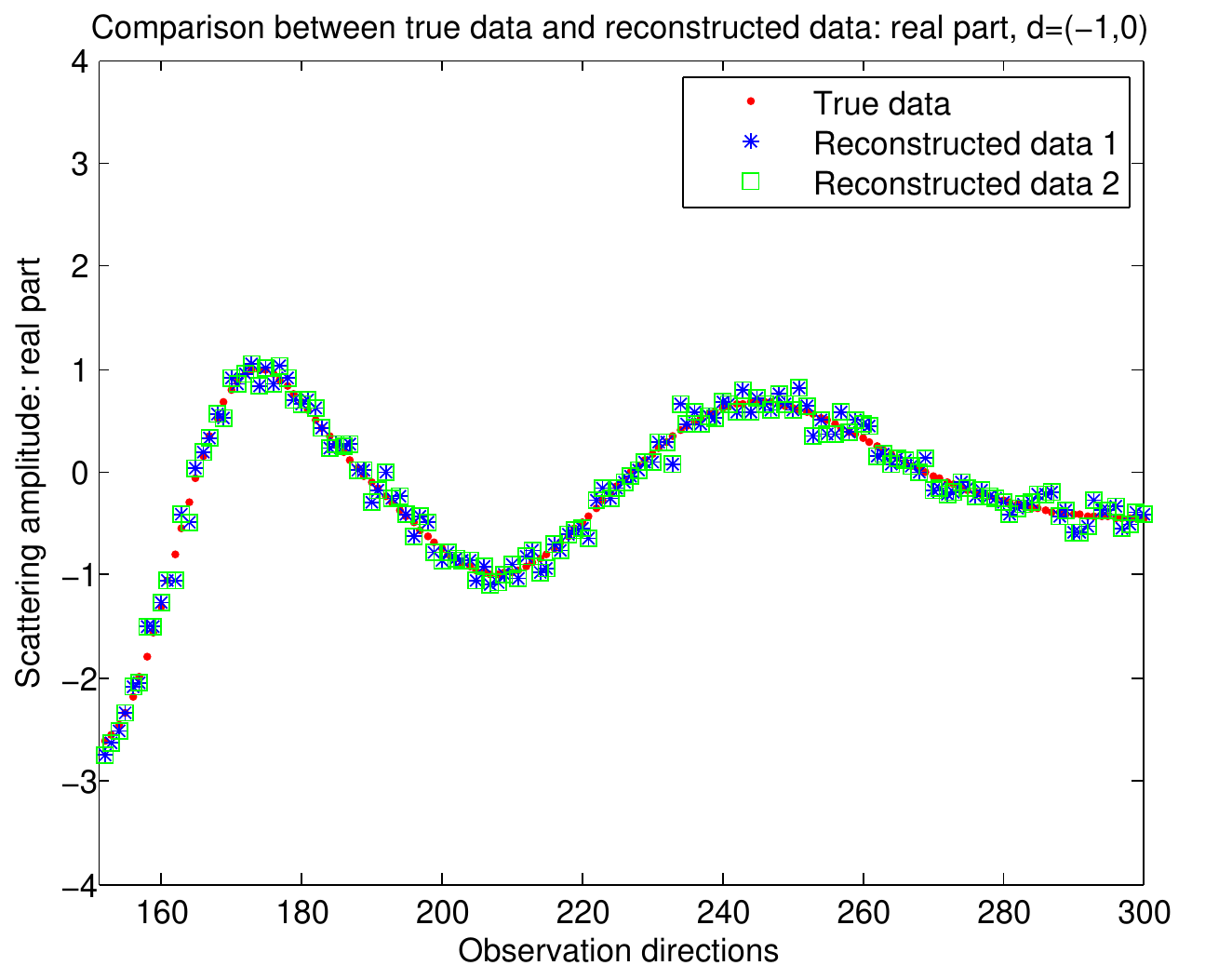}}
  \subfigure[\textbf{$d=(0,-1)$}]{
    \includegraphics[width=2.5in]{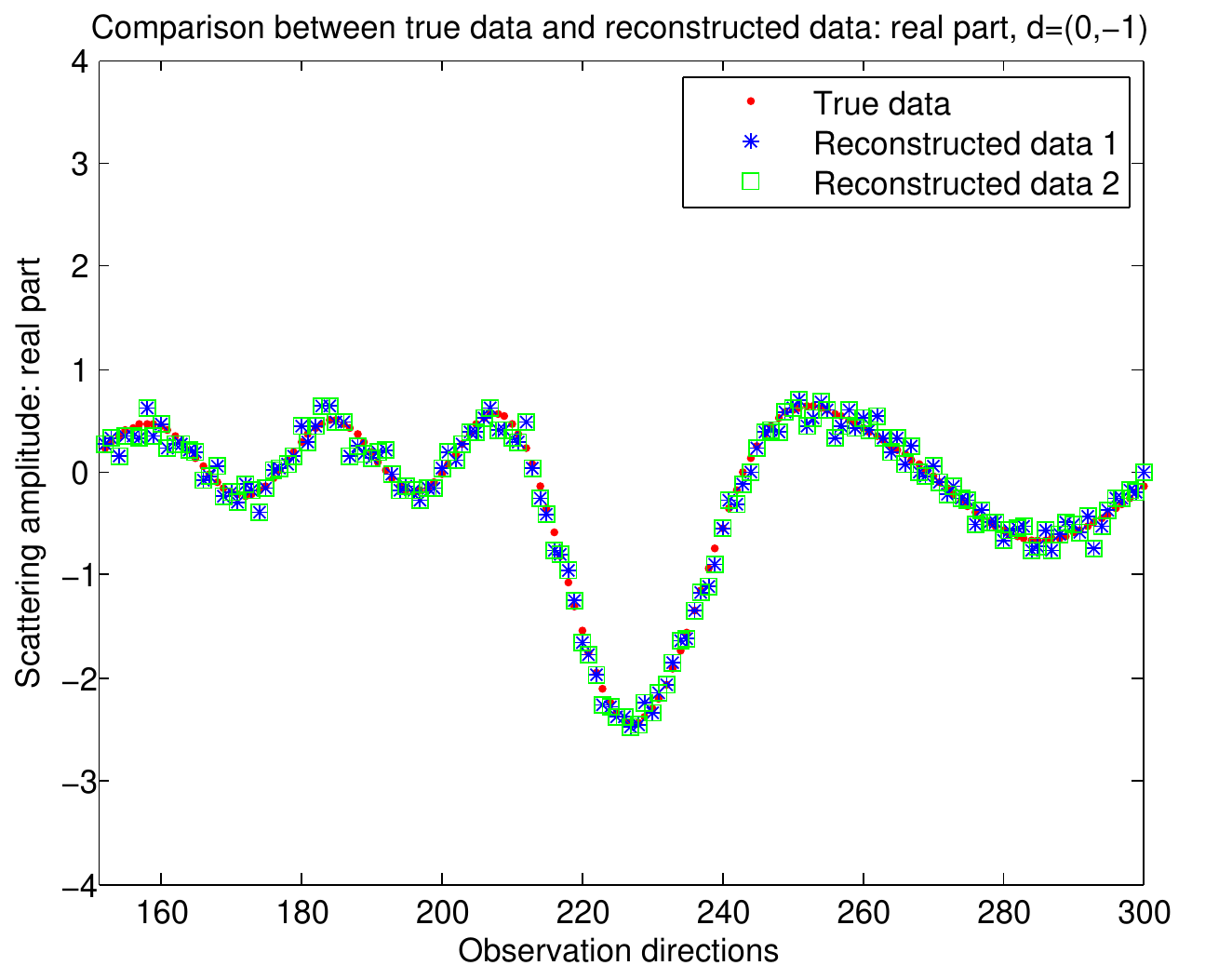}}
\caption{Exact data and recovered data with four different incident directions.
Reconstructed data $1$ is obtained by {\bf MGF}.
Reconstructed data $2$ is obtained by {\bf MSLP}.
The measurements are taken with observation angles $\phi\in (0,\pi)$. }
\label{dataretrievalkite2}
\end{figure}

\subsection{Applications in Sampling Methods}
We first test the recovered data by a novel direct sampling method ({\bf DSM}) proposed in \cite{LiuIP17}, which uses an indicator functional defined as
\be\label{I}
I(z):=|\phi(z;-d)\mathbb{F}_{full}\phi^{T}(z;\hat{x})|^2,
\en
where $\phi(z;-d):=(e^{-ikz\cdot d_1},e^{-ikz\cdot d_2},\cdots,e^{-ikz\cdot d_{2m}})$ and
$\phi(z;\hat{x}):=(e^{ikz\cdot \hat{x}_1},e^{ikz\cdot \hat{x}_2},\cdots,e^{ikz\cdot \hat{x}_{2m}})$.
The indicator takes its maximum on or near the boundary of the scatterer. Consequently, the plot of the indicator can be used to reconstruct the scatterer.
%Since only matrix multiplications are involved, it is very fast and robust against measurement noise.
%In addition, the indicator is independent of any a priori information of the unknown scatterers.
% It is expected that the indicator takes its maximum on or near the boundary of the scatterer.

This method can be modified to use only {\em limited-aperture} data by introducing
\be\label{Ilimit}
I_{limit}(z):=|\phi(z;-d)\mathbb{F}^{(l)}_{limit}\phi_{limit}^{T}(z;\hat{x})|^2,
\en
where $\phi_{limit}(z;\hat{x}):=(e^{ikz\cdot \hat{x}_1},e^{ikz\cdot \hat{x}_2},\cdots,e^{ikz\cdot \hat{x}_{l}})$
corresponds to the {\em limited-aperture} observation directions and $\mathbb{F}^{(l)}_{limit}$ is the {\em limited-aperture} data given by \eqref{MSR-l}.

%Recall the {\em limited-aperture} data $\widetilde{\mathbb{F}}^{(l)}_{limit}$ given in \eqref{MSR-2}, which is nearly exactly reconstructed by the known data $\mathbb{F}^{(l)}_{limit}$.
%Define $\widetilde{\phi}_{limit}(z;-d):=(e^{-ikz\cdot d_{m+1}},e^{-ikz\cdot d_{m+2}},\cdots,e^{-ikz\cdot d_{m+l}})$
%and $\widetilde{\phi}_{limit}(z;\hat{x}):=(e^{ikz\cdot \hat{x}_{l+1}},e^{-ikz\cdot \hat{x}_{l+2}},\cdots,e^{-ikz\cdot \hat{x}_{2m}})$.
%Then, based on the first data retrieval technique introduced in subsection 2.1,
%one may also consider the second indicator
%\be\label{I2limit}
%I^{\prime}_{limit}(z):=|\phi(z;-d)\mathbb{F}^{(l)}_{limit}\phi_{limit}^{T}(z;\hat{x})+\widetilde{\phi}_{limit}(z;-d)\widetilde{\mathbb{F}}^{(l)}_{limit}\widetilde{\phi}^{T}(z;\hat{x})|^2.
%\en
%We expect that the quality of the reconstructions can be improved by using the indicator $I^{\prime}_{limit}(z)$.

Denote by $\mathbb{F}^{(2)}_{full}$ and $\mathbb{F}^{(3)}_{full}$ the recovered {\em full-aperture} data using {\bf MGF} and {\bf MSLP}, respectively.
We introduce the following indicator
\be\label{Ifull}
I_{full}^{(ii)}(z):=|\phi(z;-d)\mathbb{F}^{(ii)}_{full}\phi^{T}(z;\hat{x})|^2, \quad ii=2,3.
\en
% One can reconstruction the scatterer by $I_{limit}(z)$.
Alternatively, one can reconstruct the scatterer by $I_{full}^{(ii)}(z)$, $ii=2,3$ using recovered {\em full-aperture} data.
As will seen shortly, the quality of the reconstructions indeed improves.

We used a grid $\mathcal{G}$ of $121\times 121$ equally spaced sampling points on some rectangle $[-6,6]\times[-6,6]$.
For each point $z \in \mathcal{G}$, we compute the indicator functionals given in \eqref{Ilimit}-\eqref{Ifull}.

The typical feature for {\em limited-aperture} problems is that the concave part cannot be reconstructed if the the observation angles
do not cover the concave part of the obstacle.
A common criterion for judging the quality of a reconstruction method is whether the concave part of the obstacle can be successfully recovered.
The resulting reconstructions by using the indicator functional $I_{limit}(z)$ with {\em limited-aperture} far field patterns are shown
in Figures \ref{samplingmethodskite}$(a)$, $(d)$, and $(g)$ for different observation apertures.
Clearly, the quality improves with the increase of observation apertures.
We also observe that the illuminated part is well constructed, but the shadow region is elongated down range.
%This is typical of {\em limited-aperture} results.

As shown in the second and third columns of Figure \ref{samplingmethodskite}, the reconstructions are indeed improved by using the data recover techniques.
In particularly, the two wings of the kite appear and the shadow region is reconstructed very well.
Considering the severe ill-posedness of the data reconstruction of an analytic function and the relative noise level $\delta=5\%$,
the target reconstructions given in Figure \ref{samplingmethodskite} are satisfactory.
Similar results are shown in Figures \ref{samplingmethodsPeanut} for the peanut.
\begin{figure}[htbp!]
  \centering
  \subfigure[\textbf{$I_{limit}(z)$}]{
    \includegraphics[width=2in]{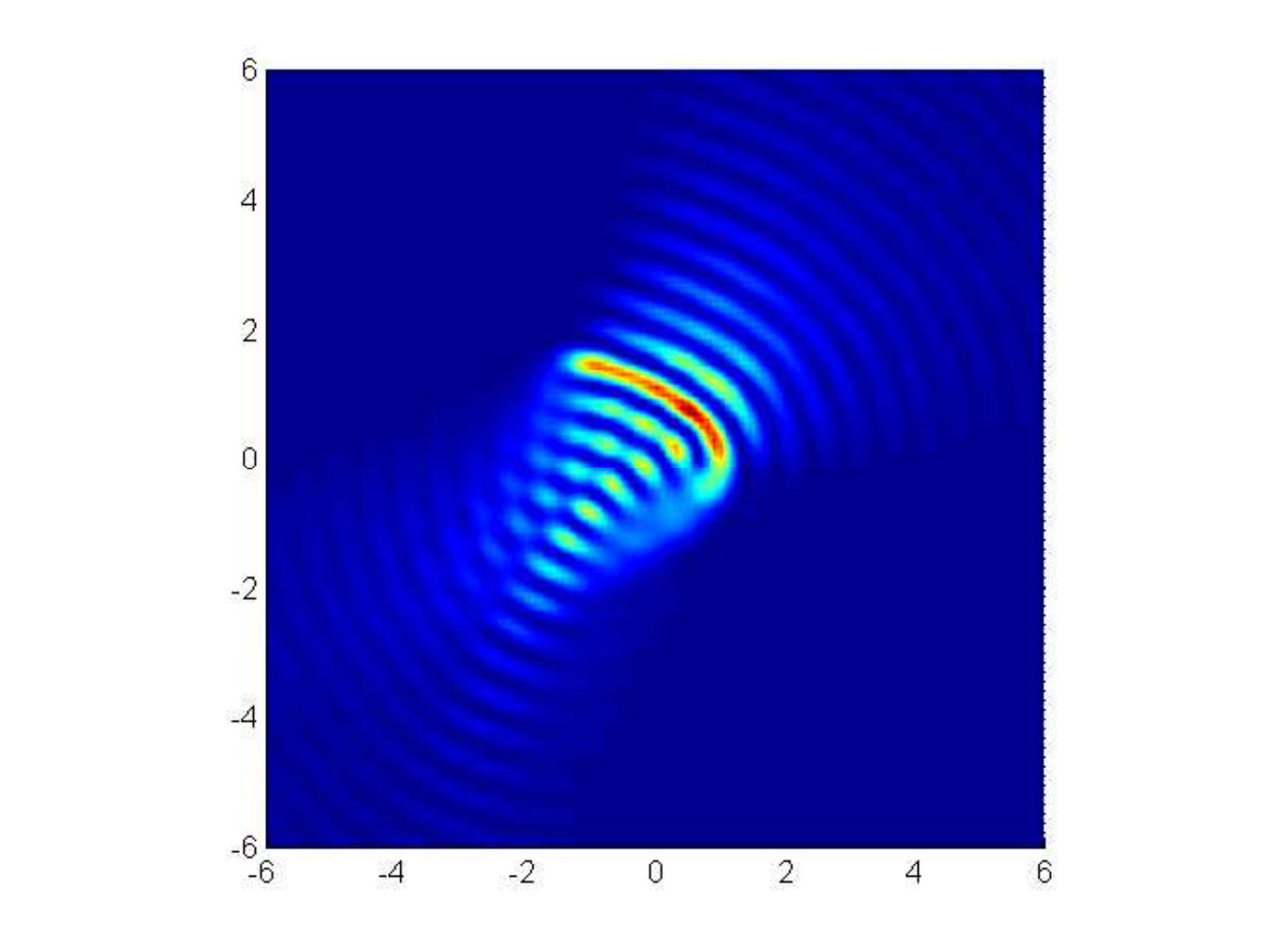}}
  \subfigure[\textbf{$I^{(2)}_{full}(z)$}]{
    \includegraphics[width=2in]{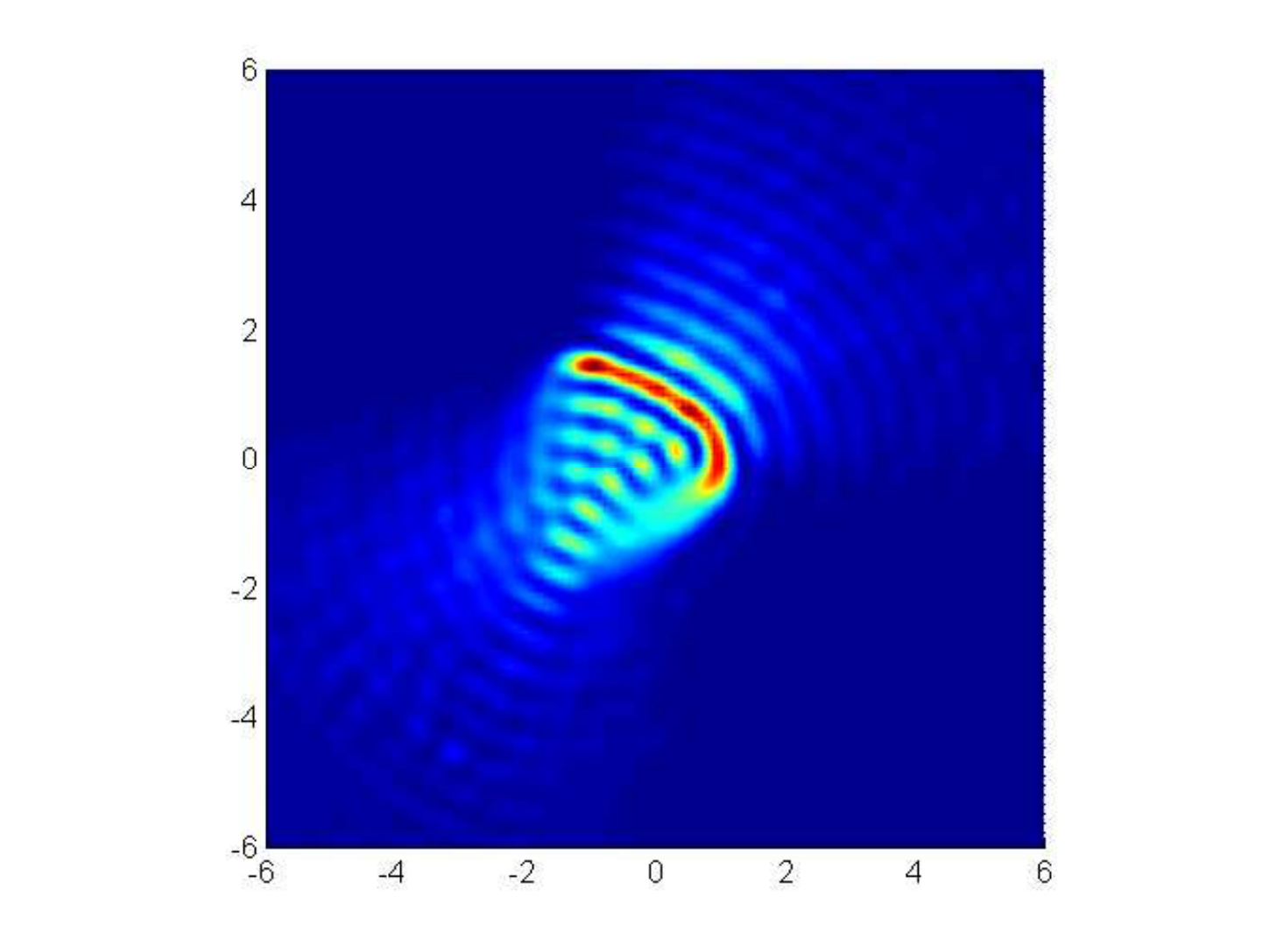}}
  \subfigure[\textbf{$I^{(3)}_{full}(z)$}]{
    \includegraphics[width=2in]{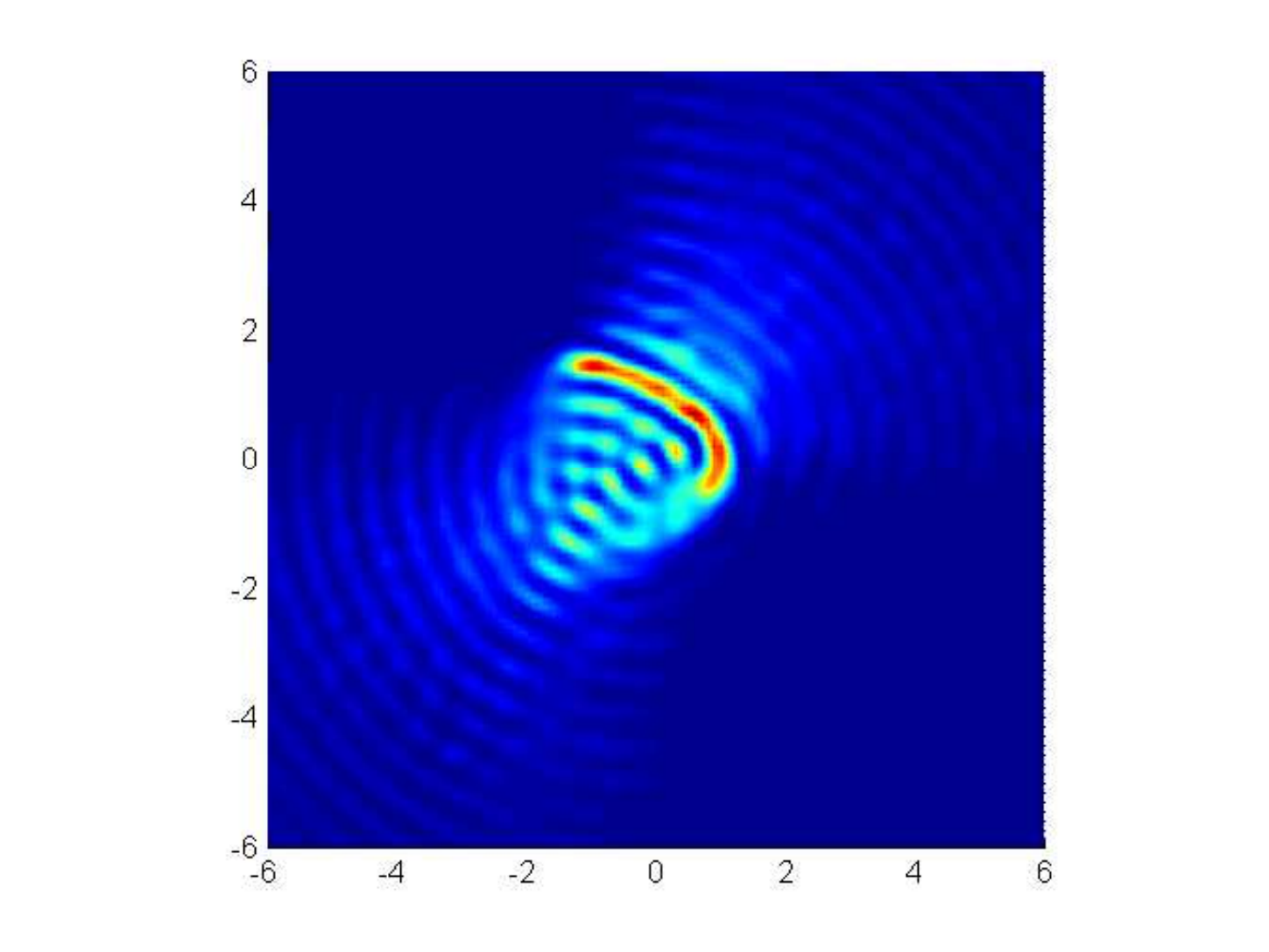}}
  \subfigure[\textbf{$I_{limit}(z)$}]{
    \includegraphics[width=2in]{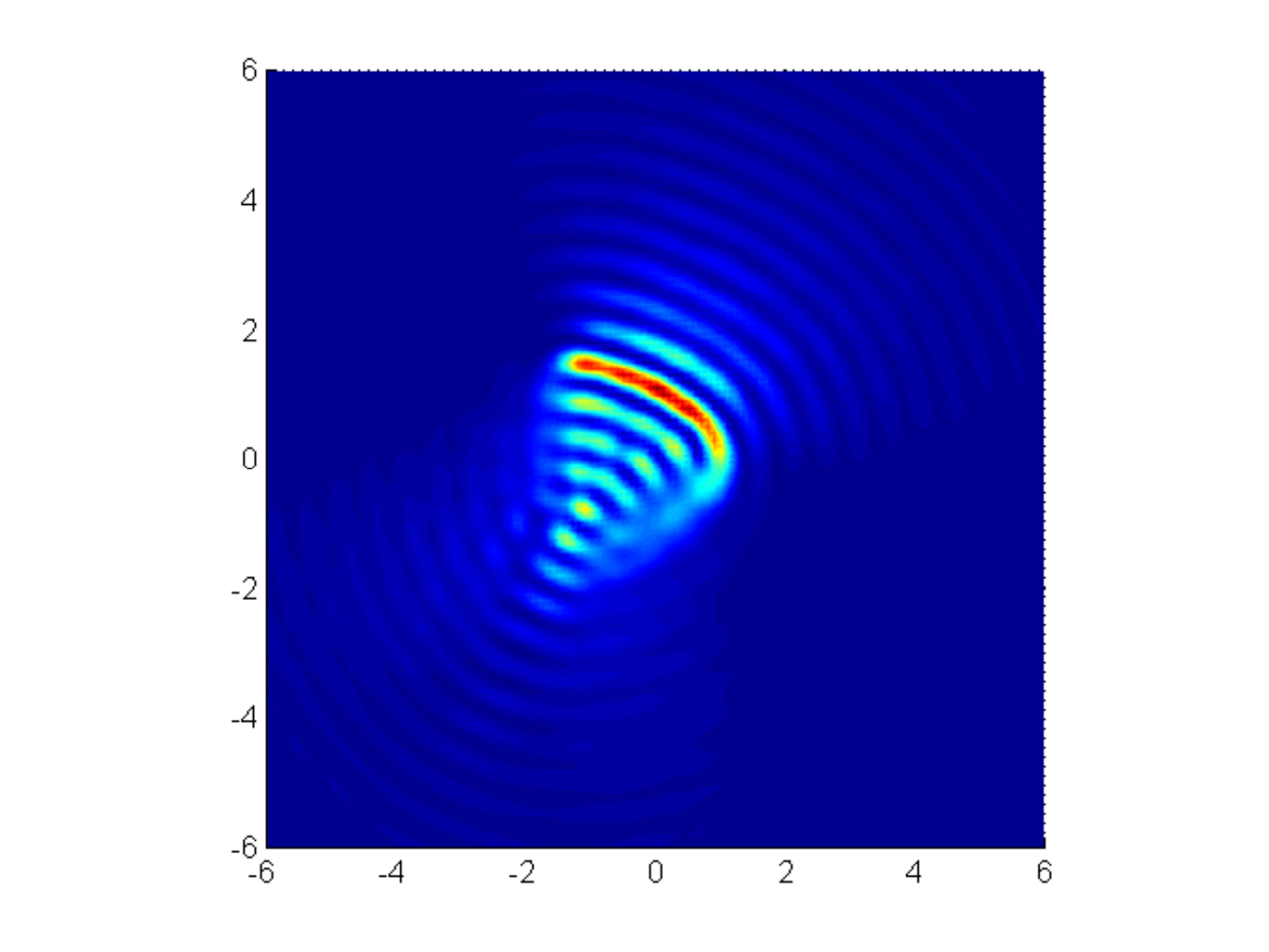}}
  \subfigure[\textbf{$I^{(2)}_{full}(z)$}]{
    \includegraphics[width=2in]{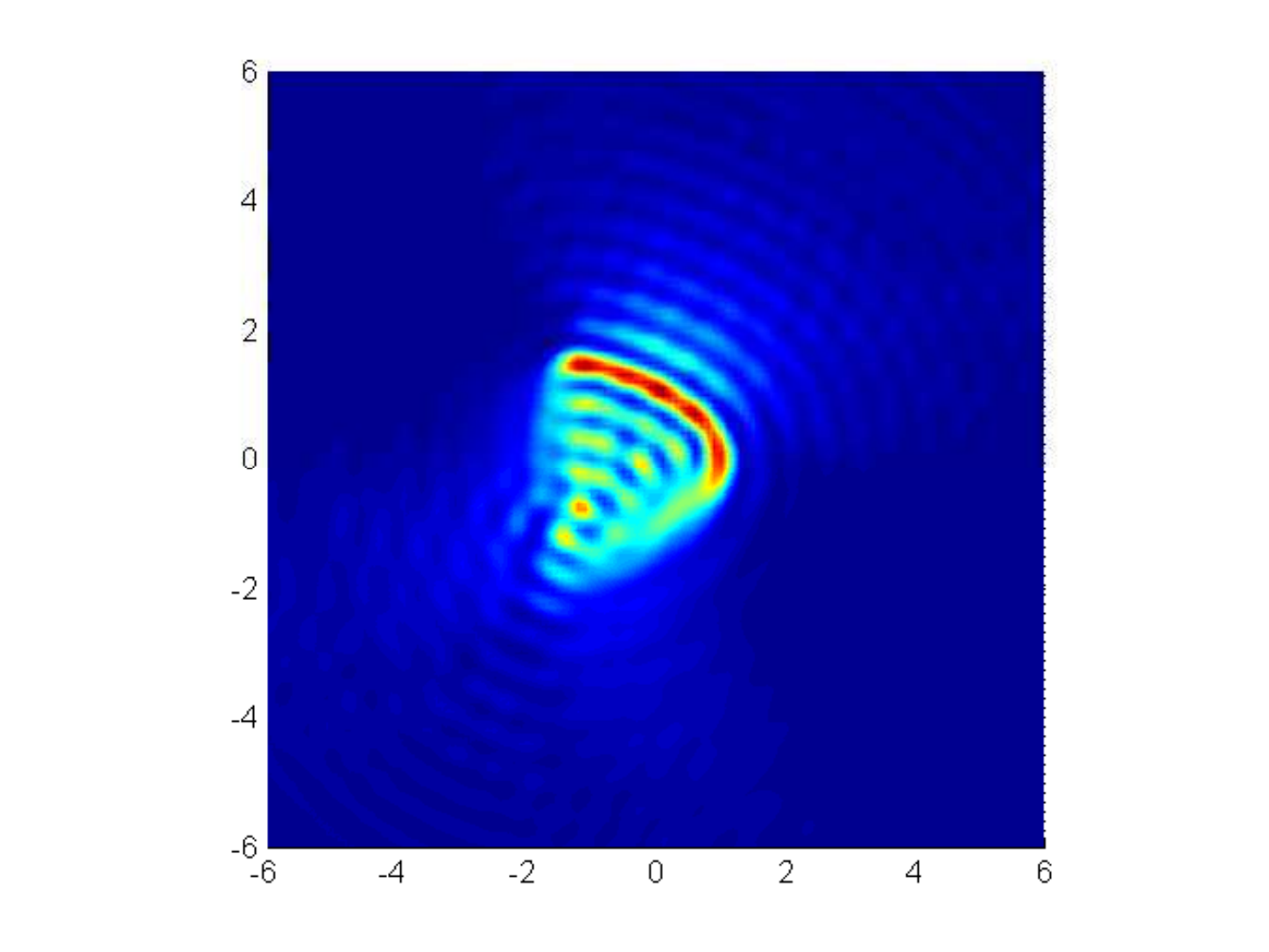}}
  \subfigure[\textbf{$I^{(3)}_{full}(z)$} ]{
    \includegraphics[width=2in]{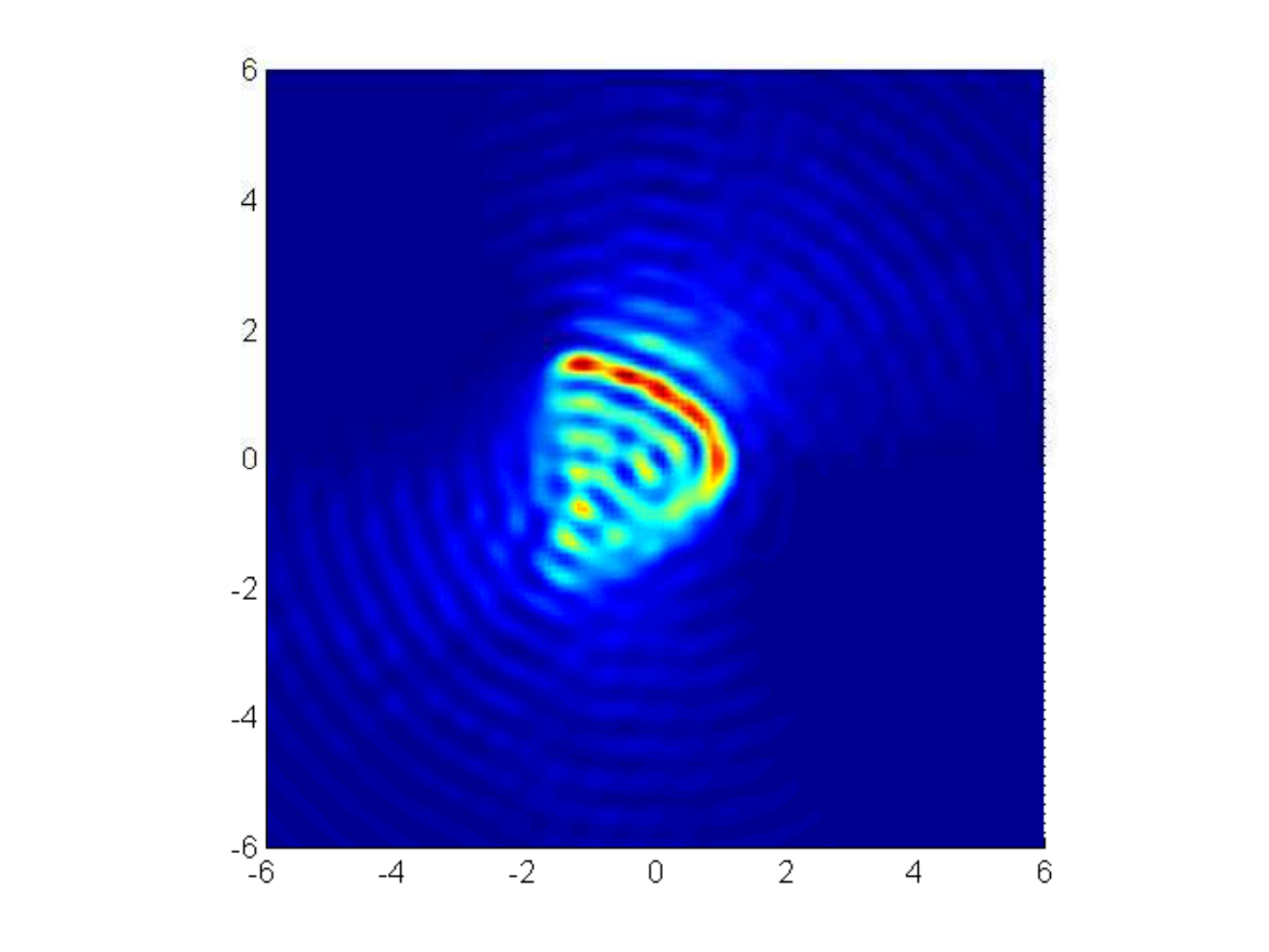}}
  \subfigure[\textbf{$I_{limit}(z)$}]{
    \includegraphics[width=2in]{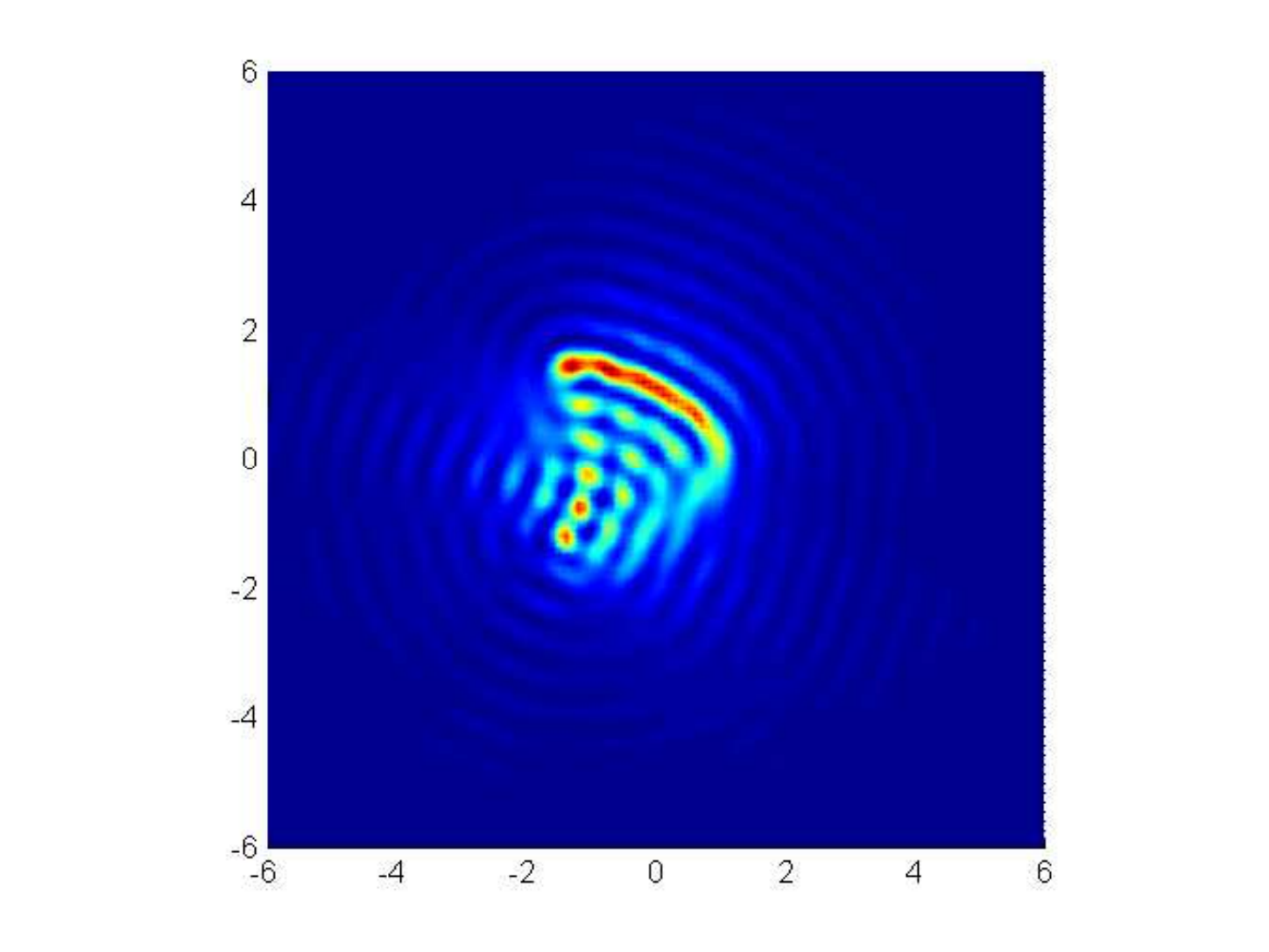}}
  \subfigure[\textbf{$I^{(2)}_{full}(z)$}  ]{
    \includegraphics[width=2in]{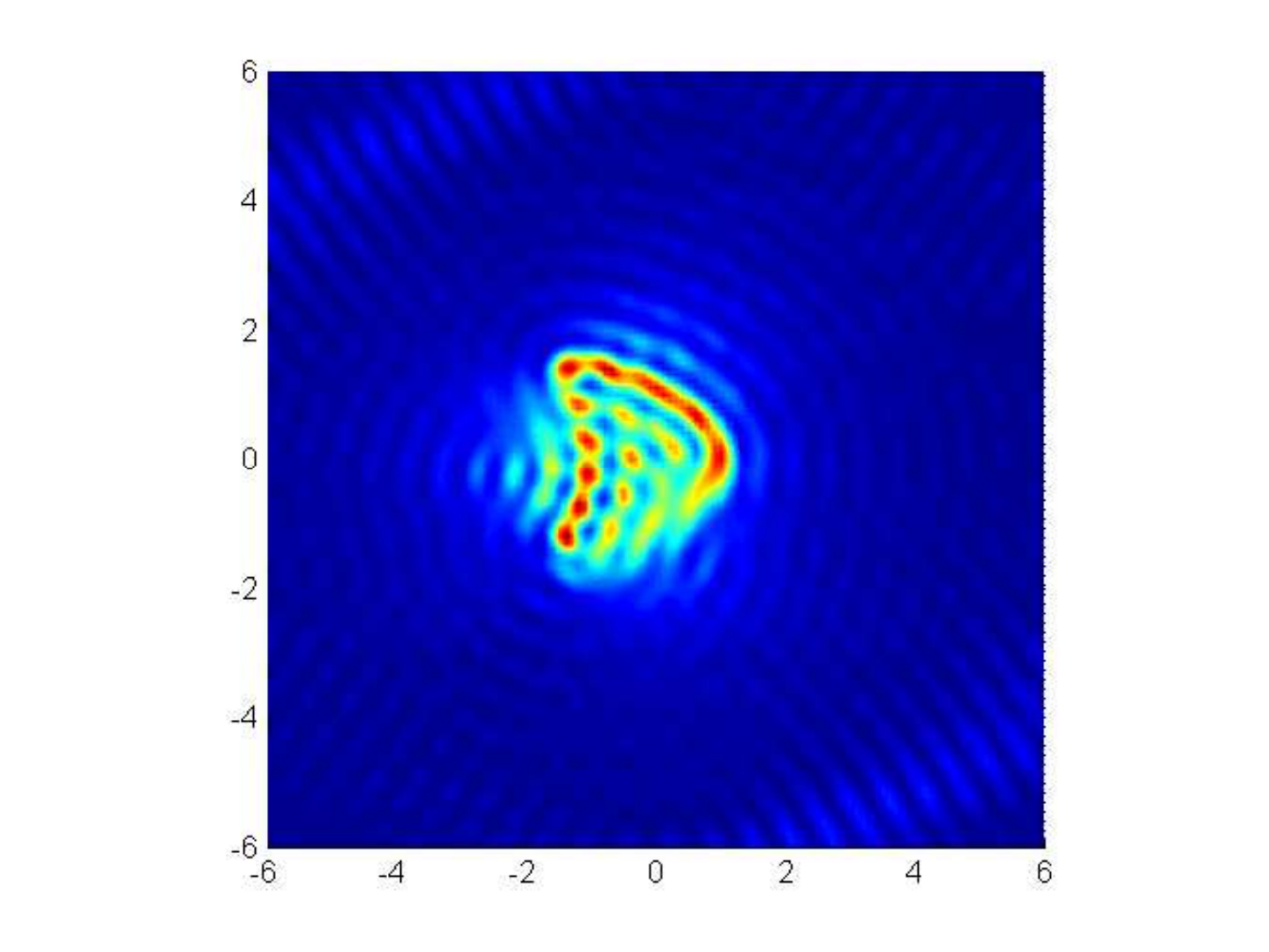}}
  \subfigure[\textbf{$I^{(3)}_{full}(z)$} ]{
    \includegraphics[width=2in]{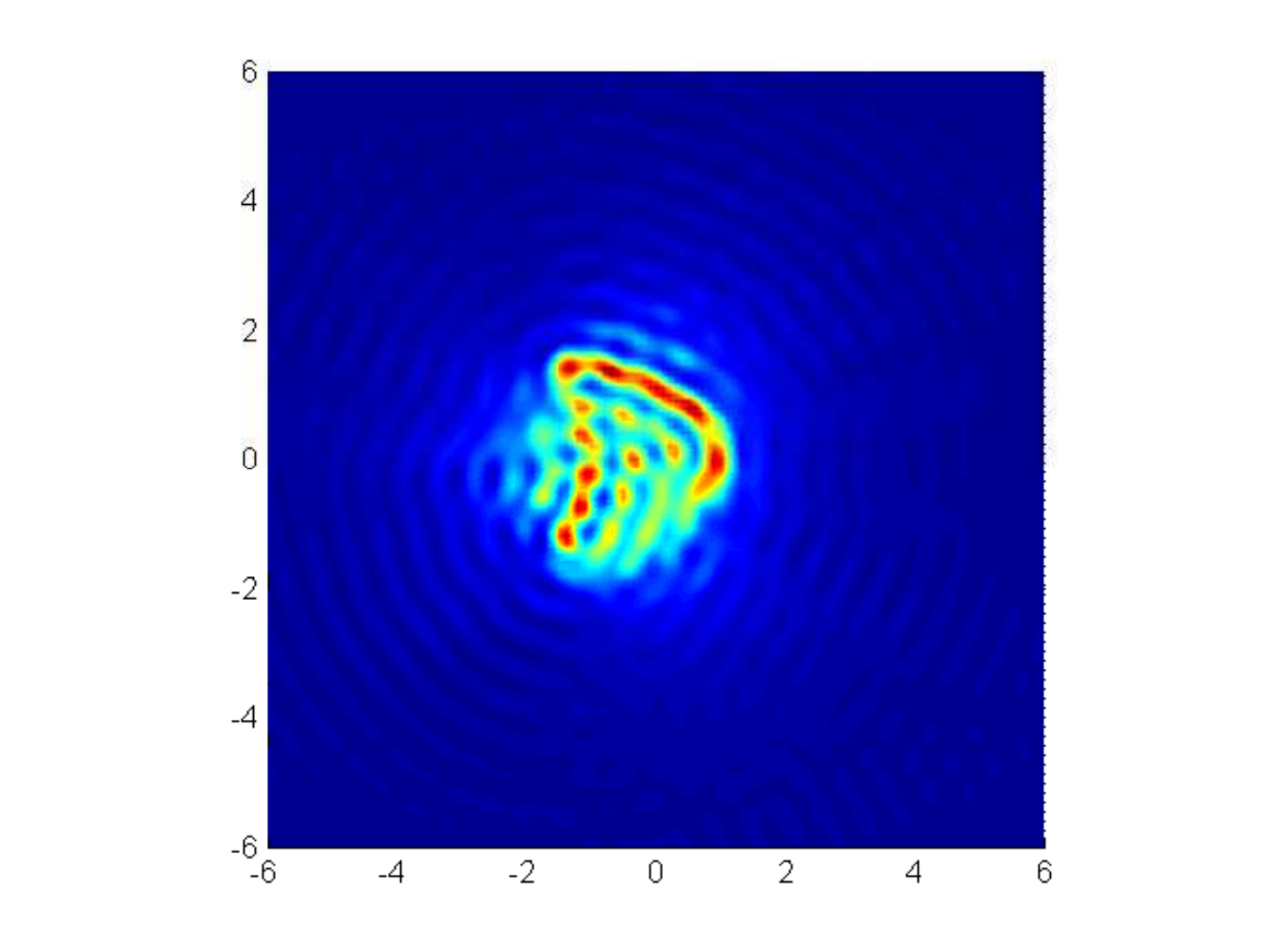}}
\caption{Shape and location reconstructions for kite by the direct sampling method.
Top row: $\phi\in (0,\pi/2)$; Middle row: $\phi\in (0,2\pi/3)$; Bellow row: $\phi\in (0,\pi)$;.}
\label{samplingmethodskite}
\end{figure}

\begin{figure}[htbp!]
  \centering
  \subfigure[\textbf{$I_{limit}(z)$}]{
    \includegraphics[width=2in]{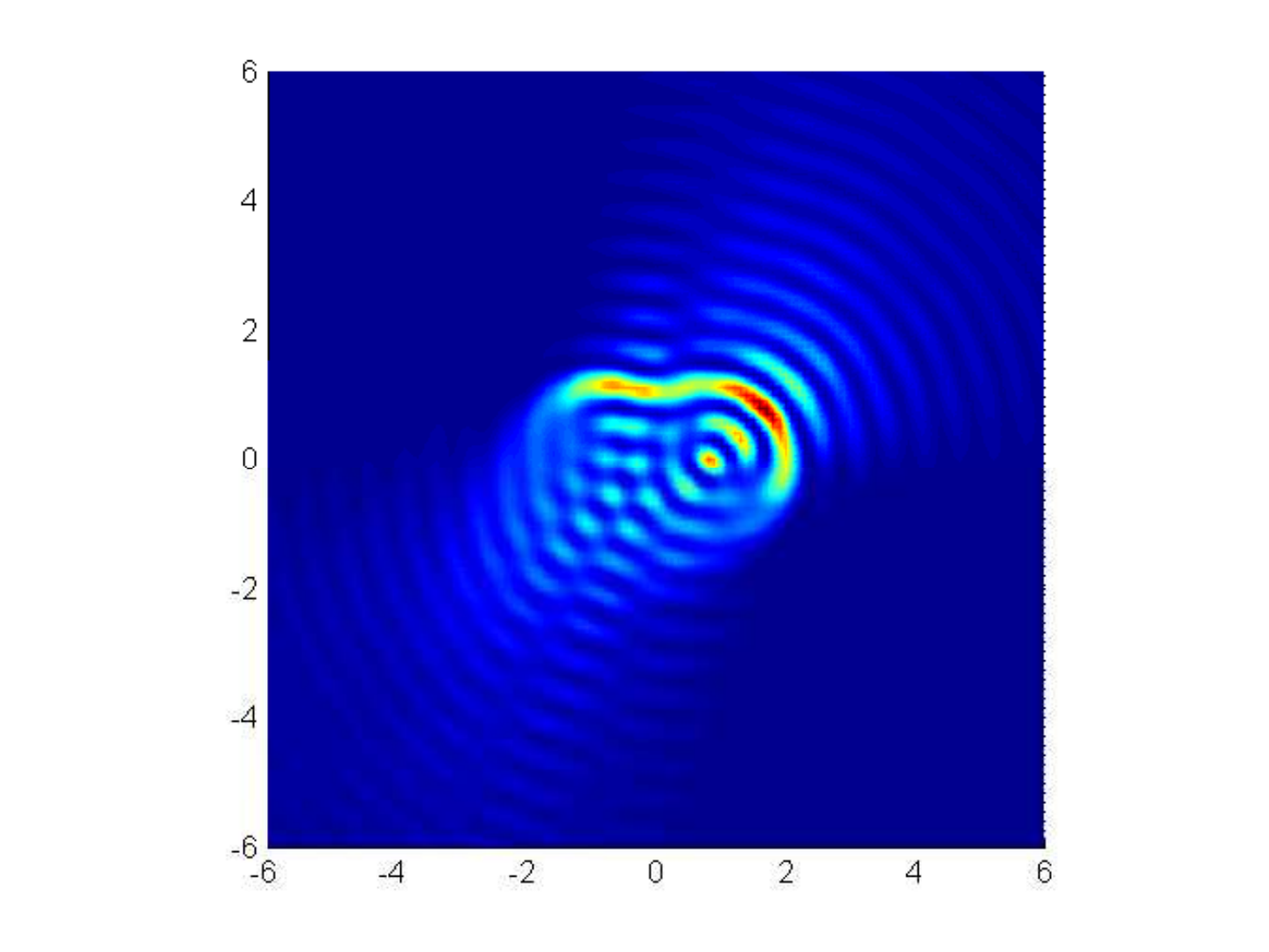}}
  \subfigure[\textbf{$I^{(2)}_{full}(z)$}  ]{
    \includegraphics[width=2in]{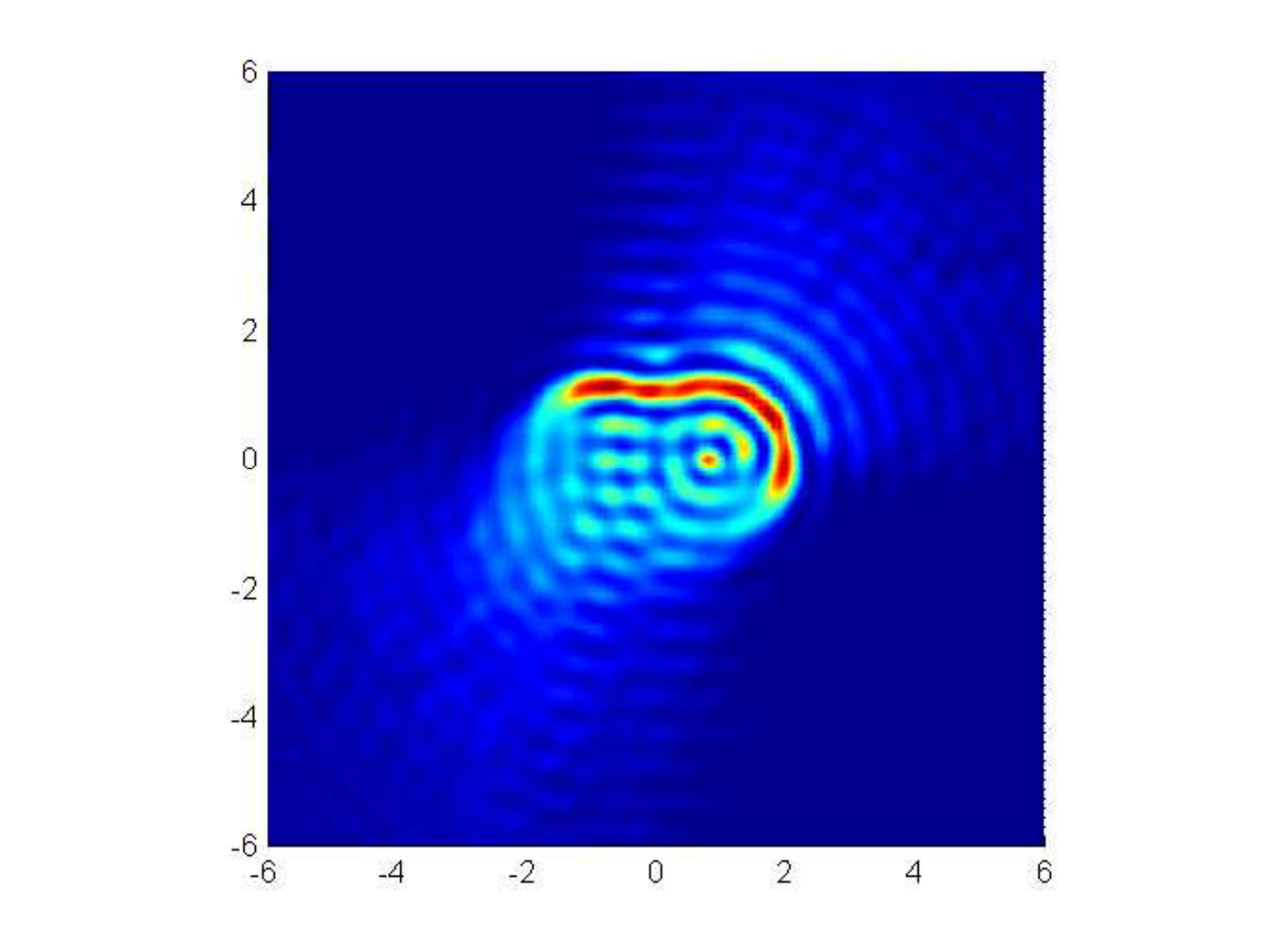}}
  \subfigure[\textbf{$I^{(3)}_{full}(z)$} ]{
    \includegraphics[width=2in]{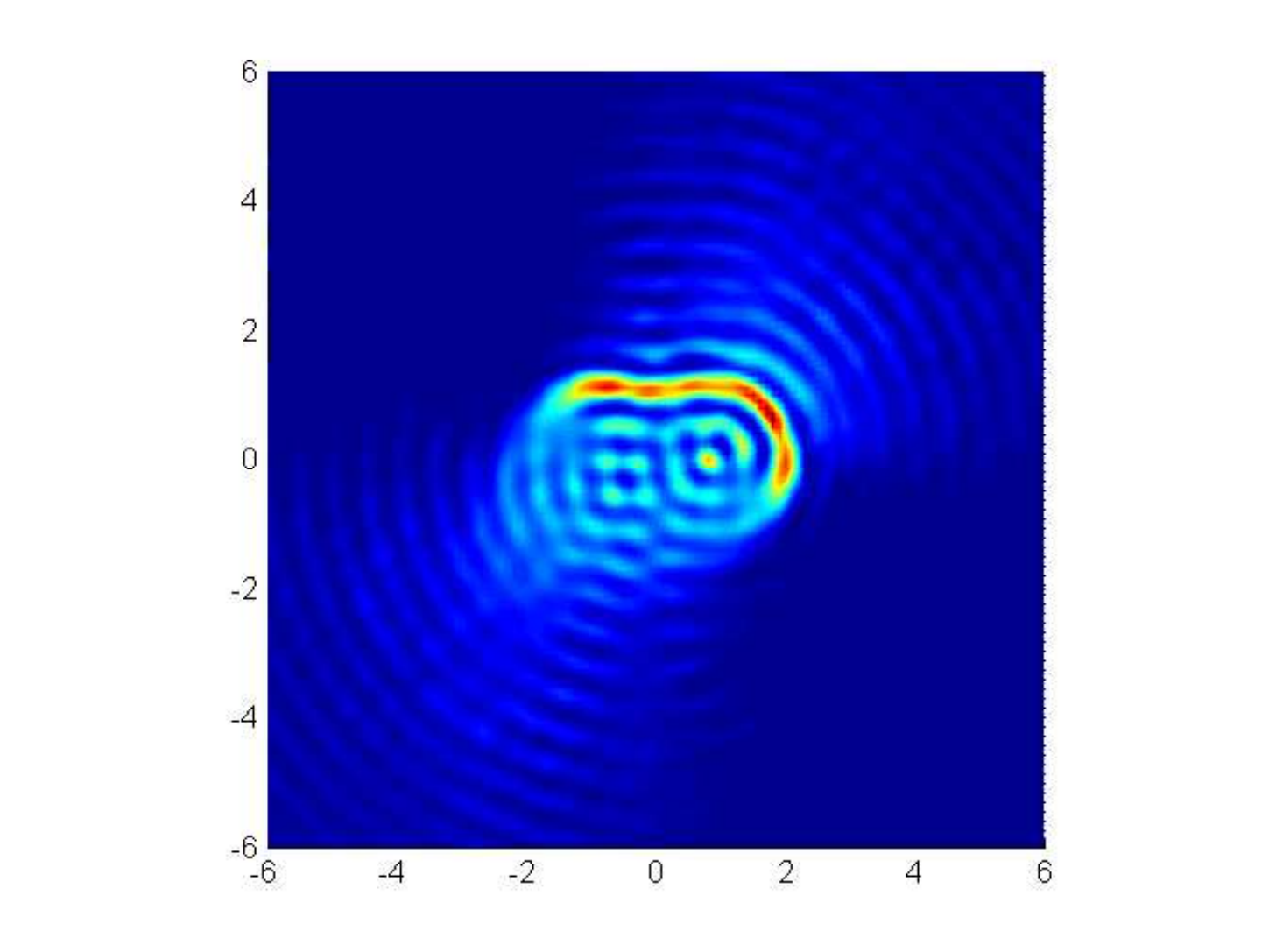}}
  \subfigure[\textbf{$I_{limit}(z)$}]{
    \includegraphics[width=2in]{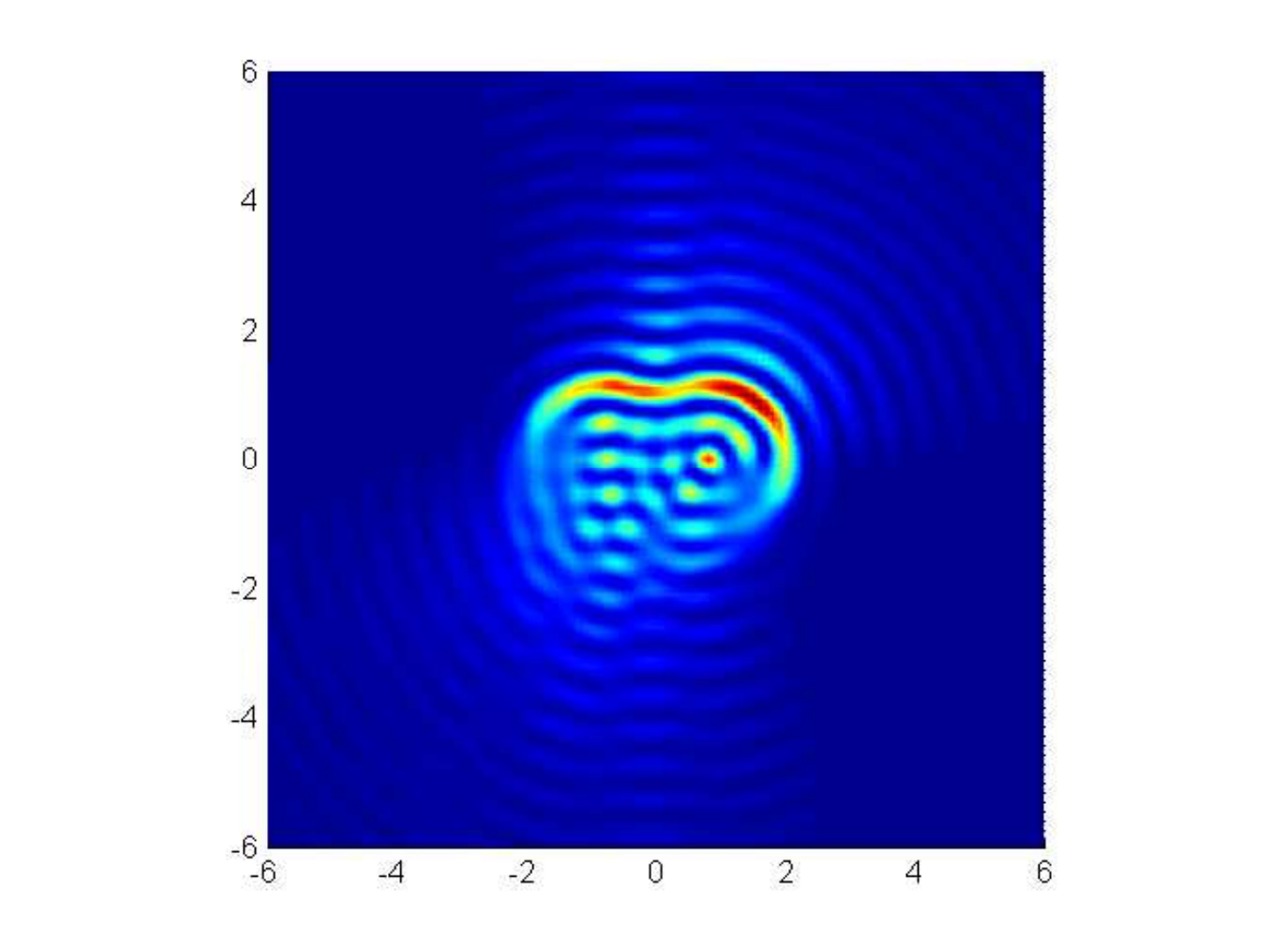}}
  \subfigure[\textbf{$I^{(2)}_{full}(z)$}  ]{
    \includegraphics[width=2in]{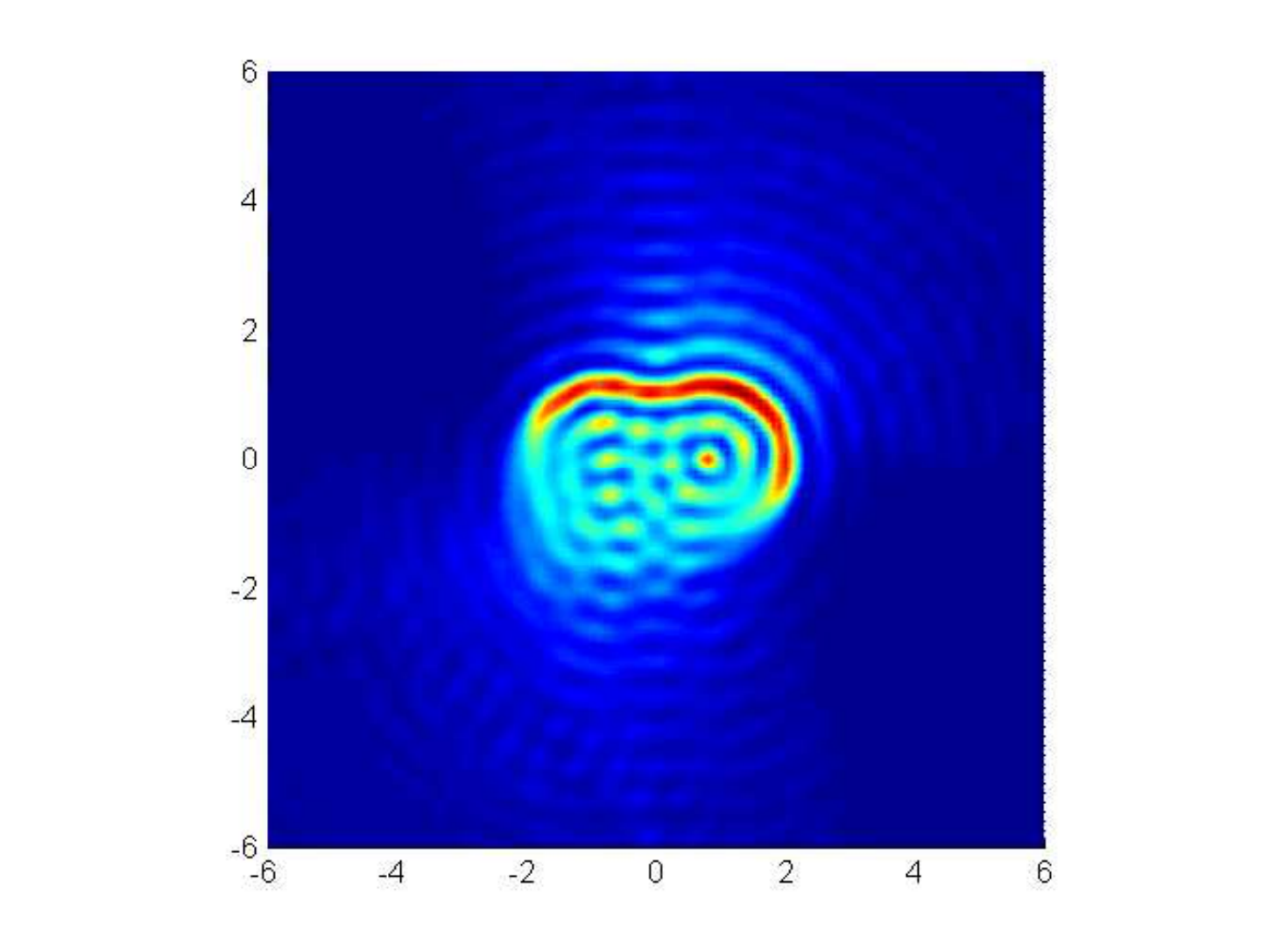}}
  \subfigure[\textbf{$I^{(3)}_{full}(z)$} ]{
    \includegraphics[width=2in]{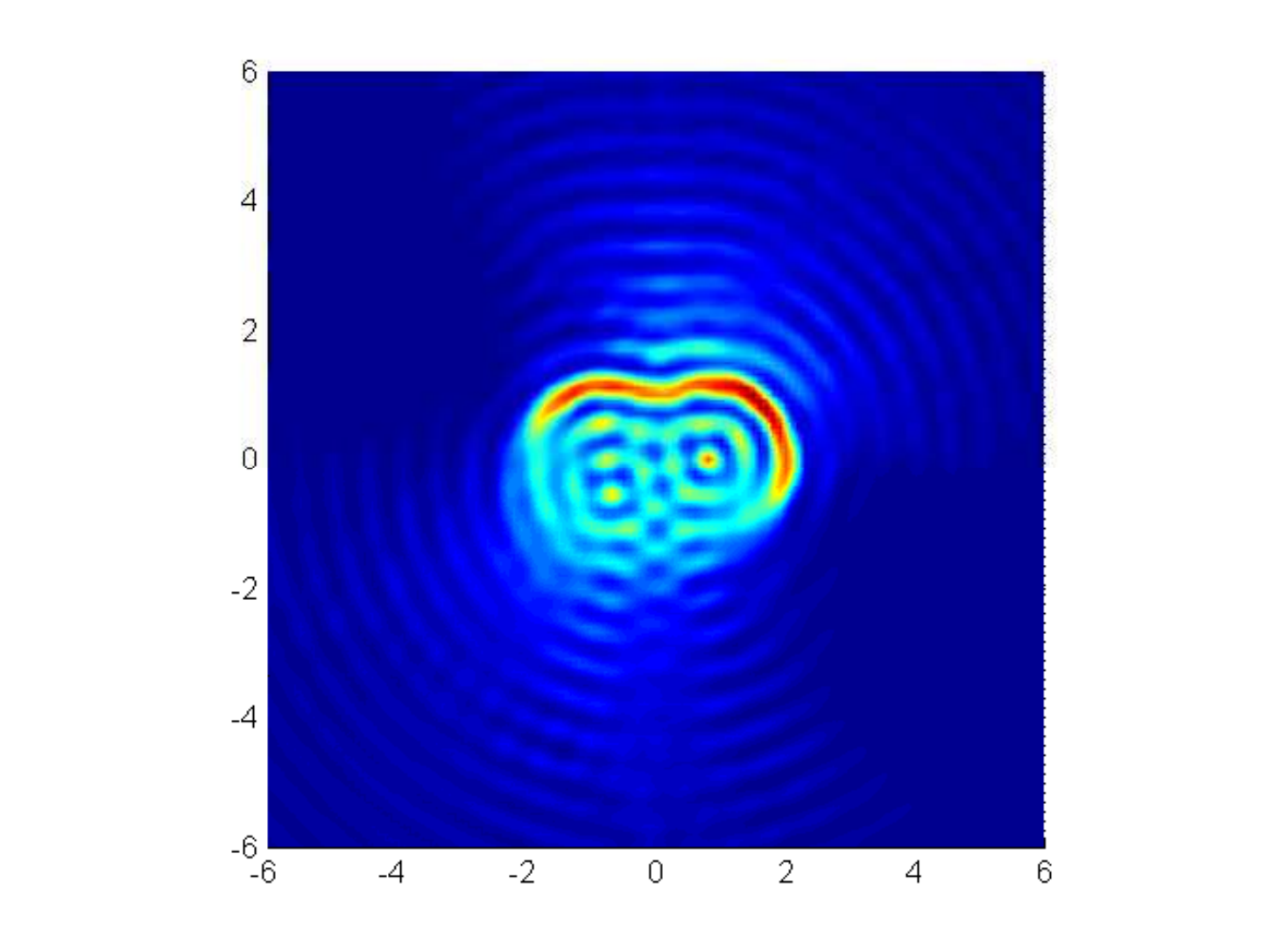}}
  \subfigure[\textbf{$I_{limit}(z)$}]{
    \includegraphics[width=2in]{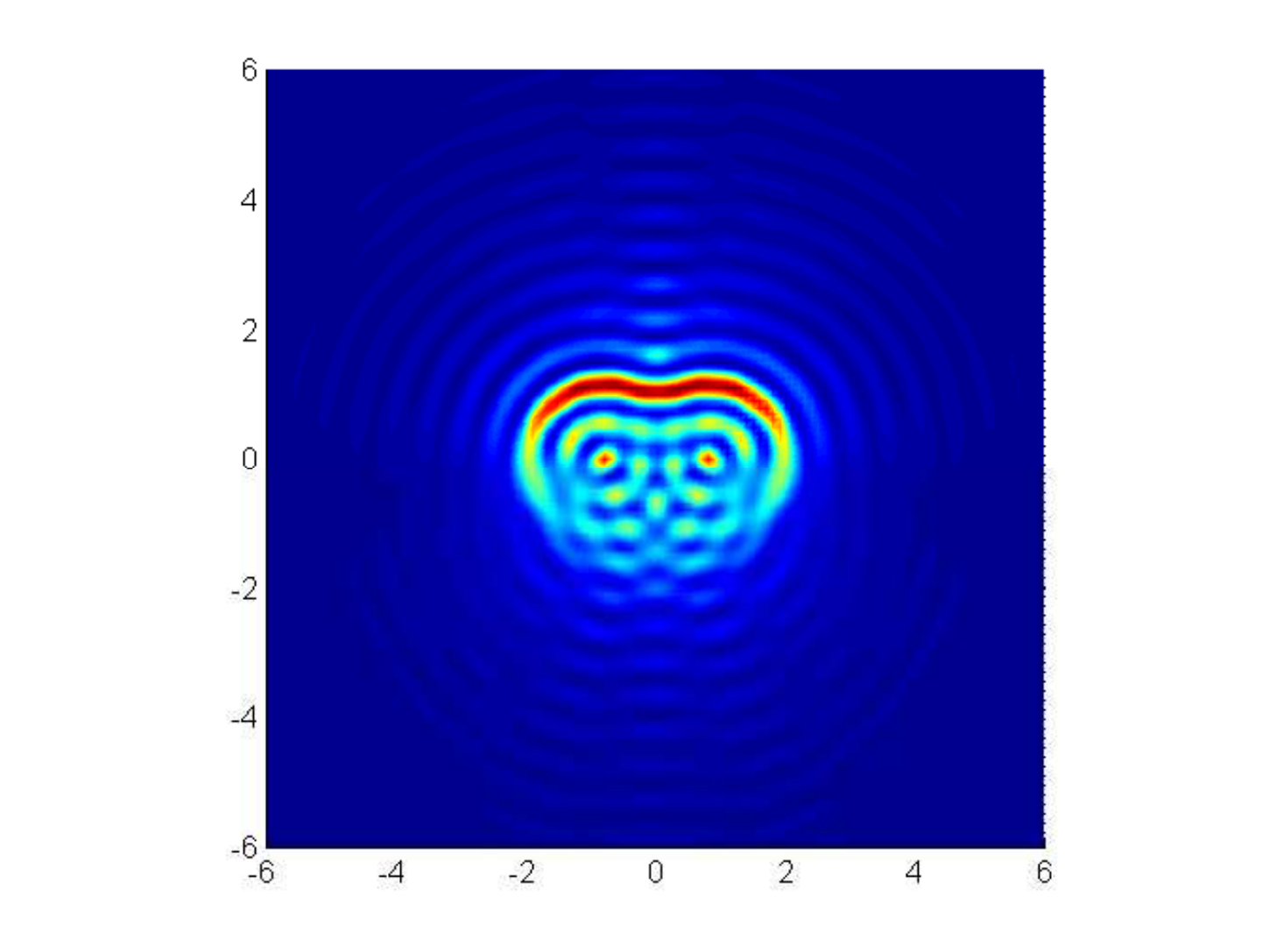}}
  \subfigure[\textbf{$I^{(2)}_{full}(z)$}  ]{
    \includegraphics[width=2in]{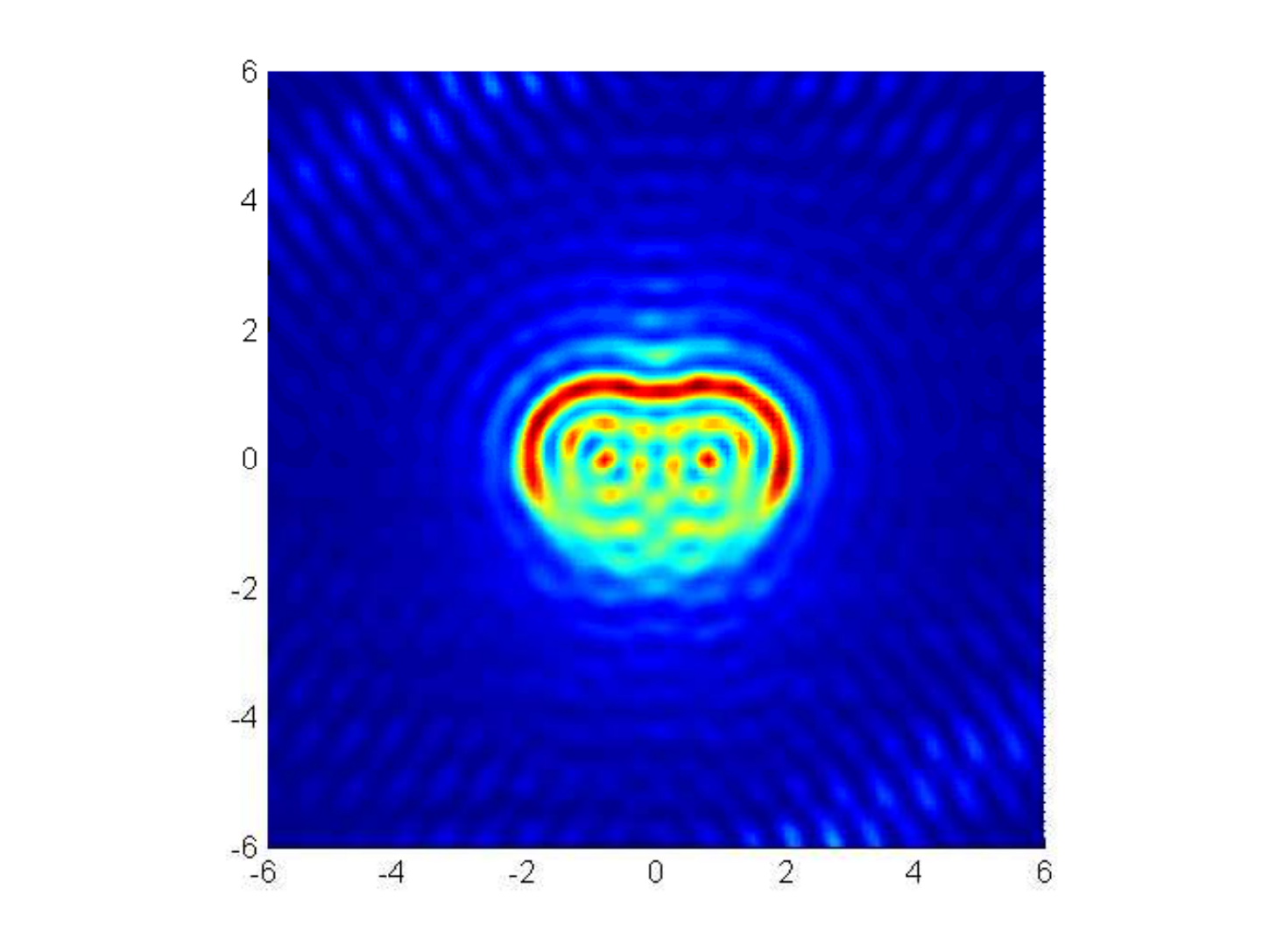}}
  \subfigure[\textbf{$I^{(3)}_{full}(z)$} ]{
    \includegraphics[width=2in]{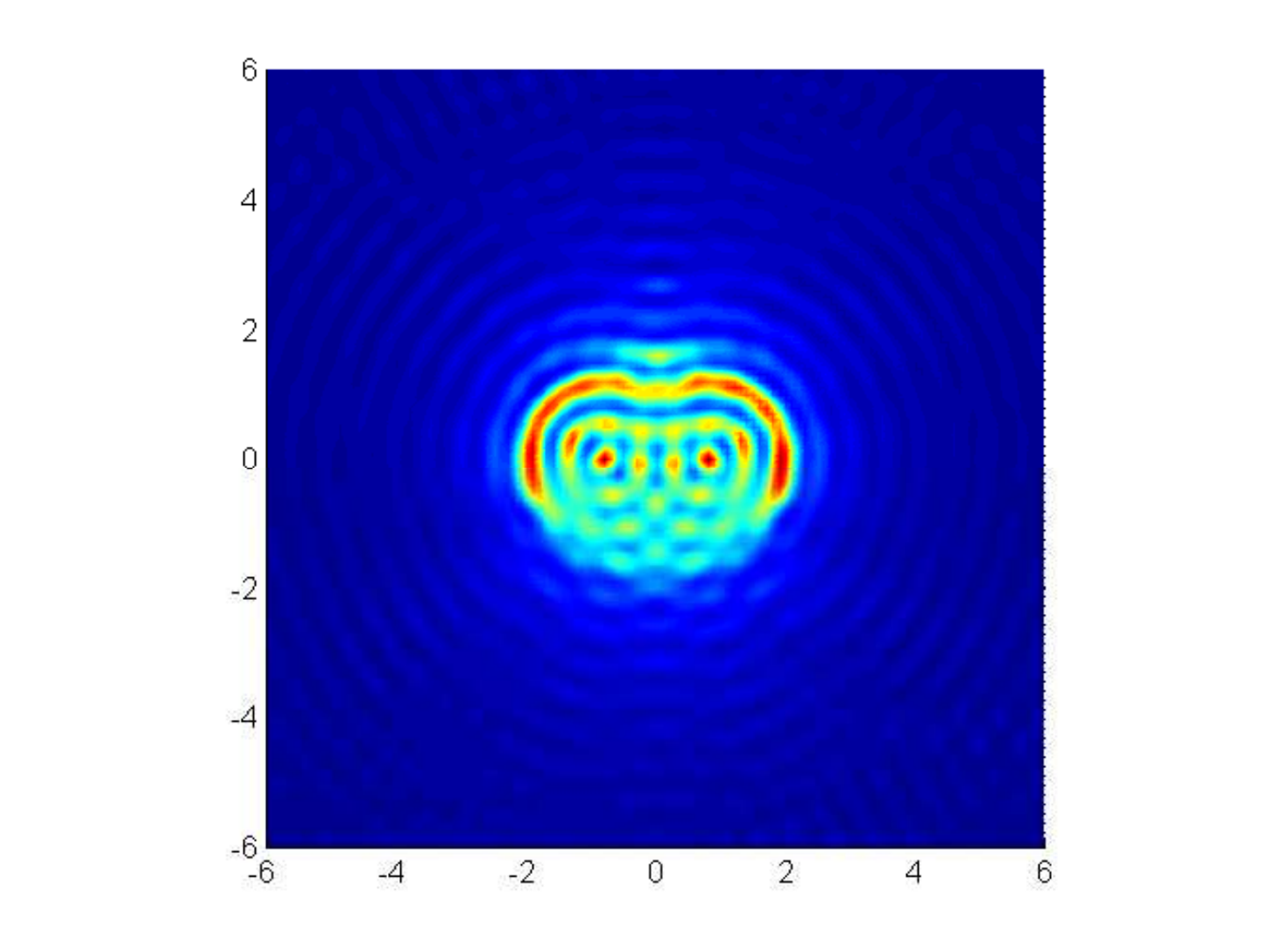}}
\caption{Shape and location reconstructions for peanut by the direct sampling method.
Top row: $\phi\in (0,\pi/2)$; Middle row: $\phi\in (0,2\pi/3)$; Bellow row: $\phi\in (0,\pi)$;.}
\label{samplingmethodsPeanut}
\end{figure}

Finally, the recovered data is tested using the factorization method proposed by Kirsch \cite{Kirsch98, KirschGrinberg}. To our knowledge, the factorization method has not been established for limited aperture data. However, the factorization method applies using the recovered {\em full-aperture} data $\mathbb{F}^{(ii)}_{full}$, $ii=2,3$.
Let $\{(\sigma_n, \psi_n): n = 1,...,2m \}$ represent the eigensystem of the matrix $F_\sharp$ given by
\ben
F_\sharp := |\Re (\mathbb{F}^{(ii)}_{full})| + |\Im (\mathbb{F}^{(ii)}_{full})|.
\enn
We then define the indicator function
\ben
I^{(ii)}_{FM}(z) : = \Biggl[\sum_{n=1}^{2m} \frac{|\phi_z^* \psi_n|^2}{|\sigma_n|}\Biggr]^{-1},\quad ii=2,3,
\enn
where
$\phi_z = (e^{-ik\theta_1 \cdot z}, e^{-ik\theta_2 \cdot z}, \dots, e^{-ik\theta_{2m} \cdot z})^\top \in  \C^{2m}$.
Although the sum is finite, we expect the values of $I_{FM}(z)$ to be much smaller for the points belonging to $\Om$ than for those lying in the exterior $\R^2\ba\ov{D}$.
Figure \ref{FMkite} shows the corresponding results with respect to measurements in $(0, \pi/2)$.

\begin{figure}[htbp!]
  \centering
  \subfigure[\textbf{$I^{(2)}_{FM}(z)$}]{
    \includegraphics[width=2in]{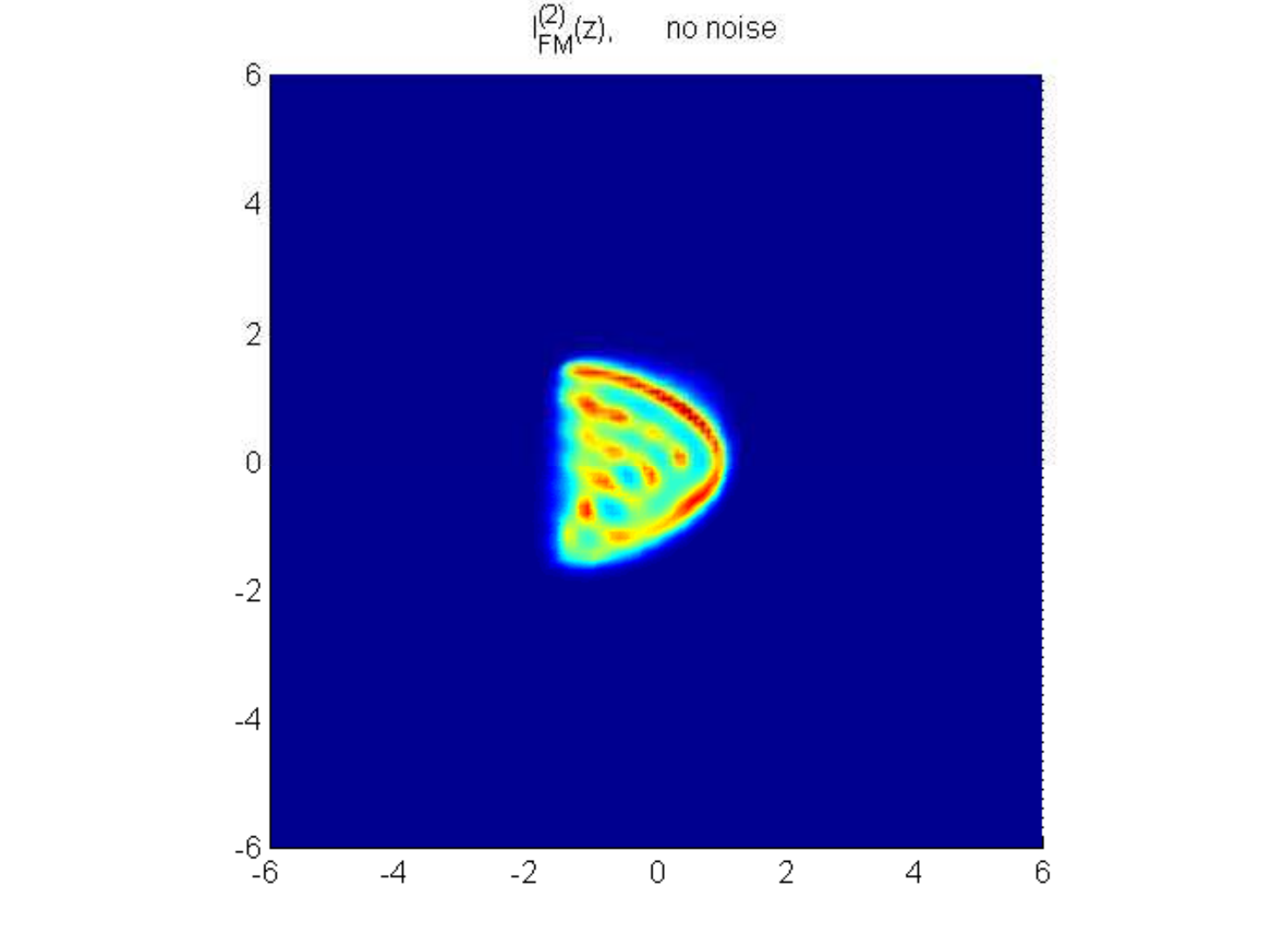}}
  \subfigure[\textbf{$I^{(3)}_{FM}(z)$}]{
    \includegraphics[width=2in]{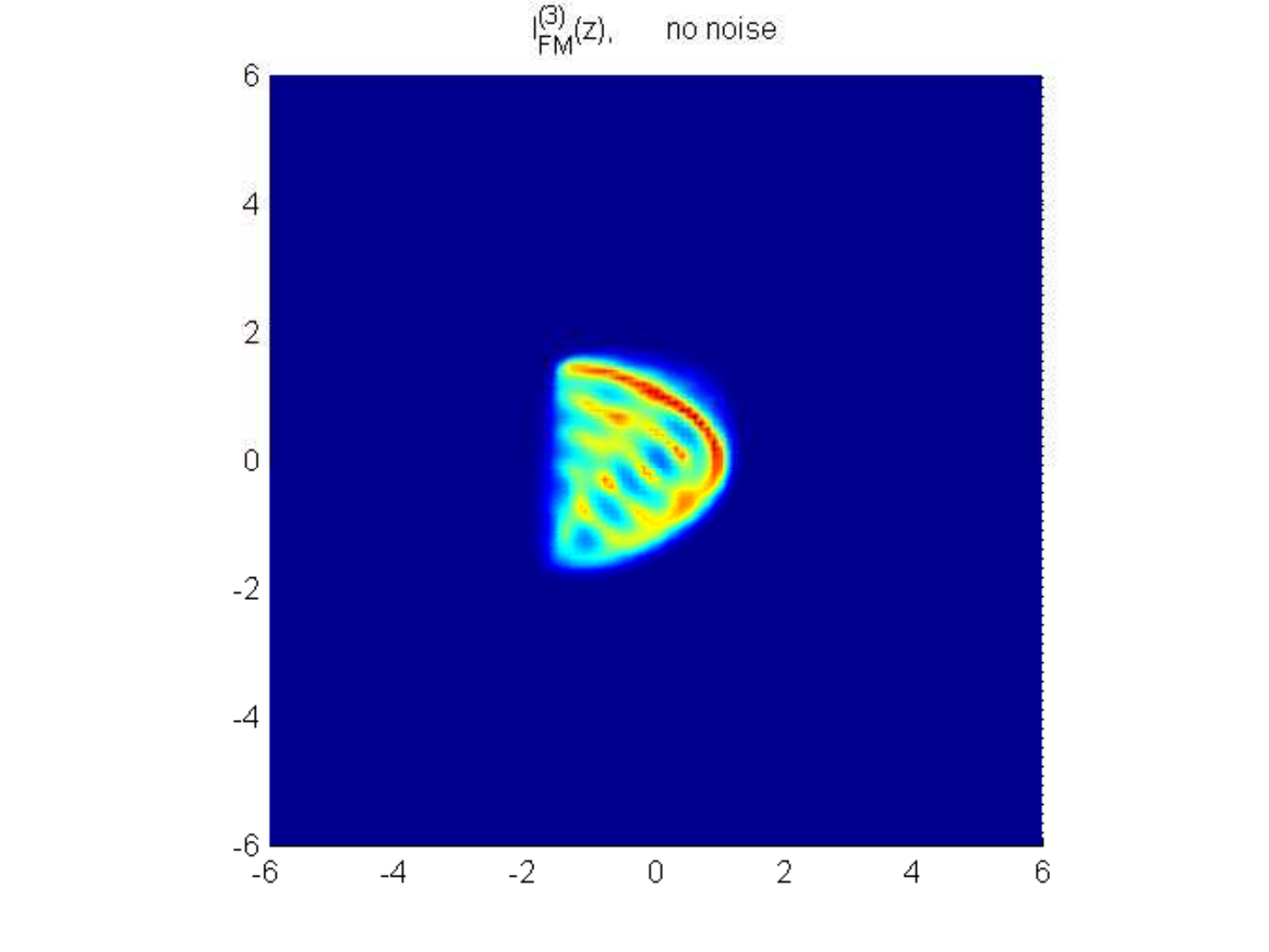}}\\
  \subfigure[\textbf{$I^{(2)}_{FM}(z)$}]{
    \includegraphics[width=2in]{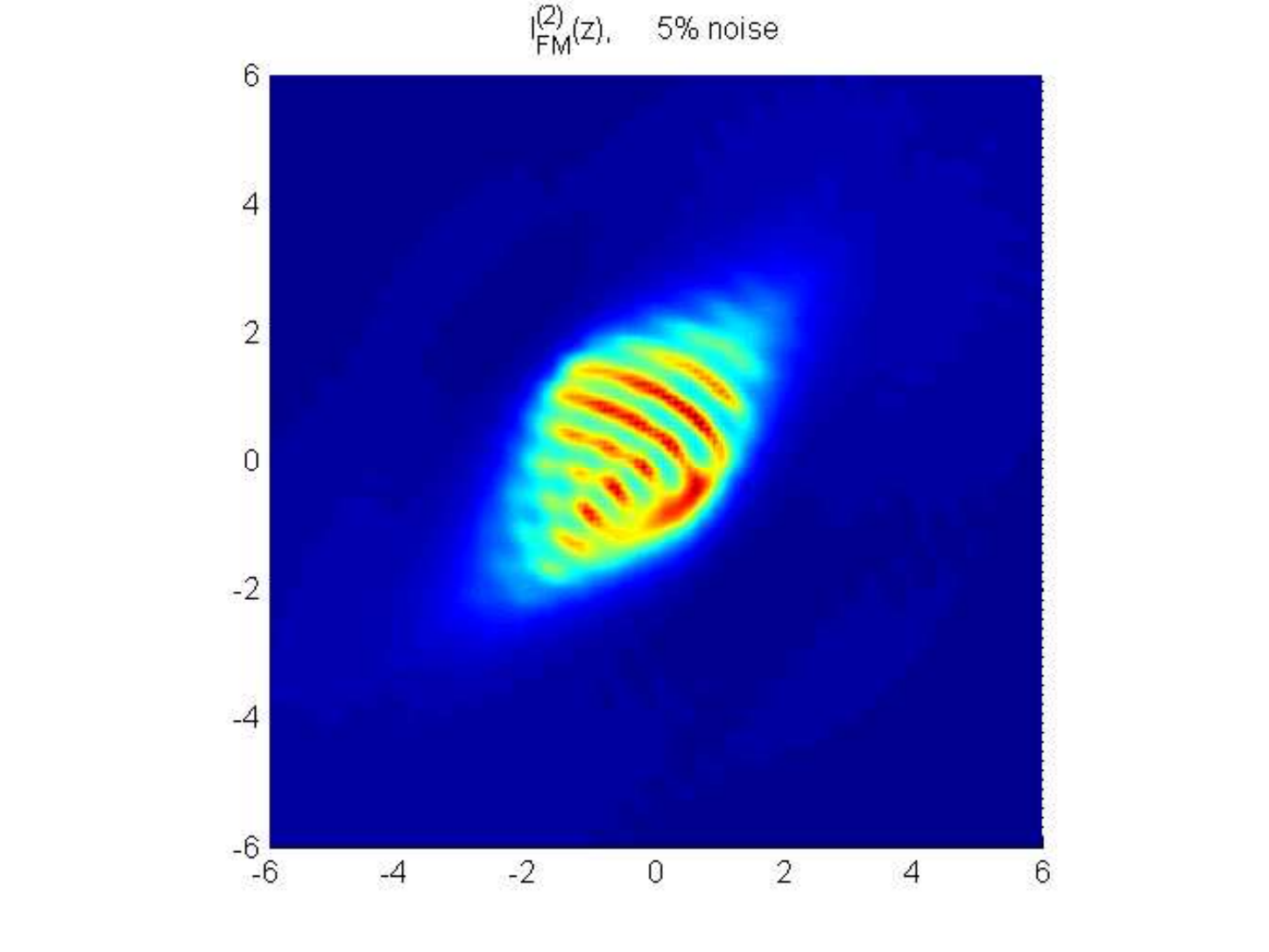}}
  \subfigure[\textbf{$I^{(3)}_{FM}(z)$}]{
    \includegraphics[width=2in]{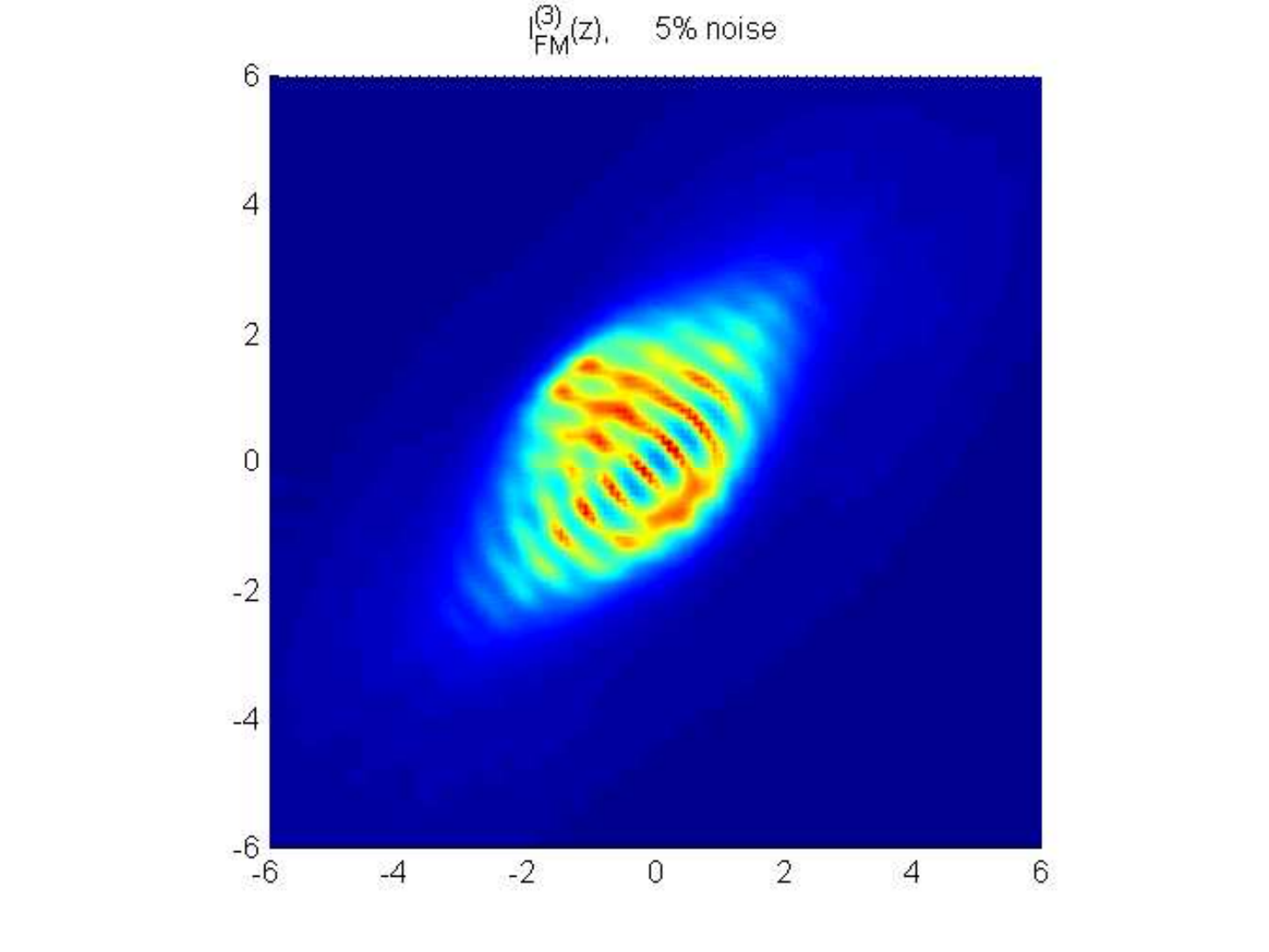}}
\caption{Reconstructions by the factorization method using recovered data. The observation angle $\phi\in (0,\pi/2)$. (a)-(b): without noise; (c)-(d): 5\% noise.}
\label{FMkite}
\end{figure}

\section{Concluding Remarks}
\label{sec5}
\setcounter{equation}{0}

The {\em limited-aperture} problems arise in various areas of practical applications such as radar, remote sensing, geophysics, and nondestructive testing.
It is well known that the illuminated part can be reconstructed well, while the shadow domain fails to be recovered.
In this paper, based on the PDE theory for scattering problems, we introduce two techniques to recover the missing data that can not be measured directly.
Using the recovered {\em full-aperture} data, a direct sampling method proposed in a recent paper \cite{LiuIP17} and the factorization method
yield satisfactory reconstructions.

We conclude with some future works.
\begin{itemize}
  \item The data recover techniques need to solve the ill-posed equations.
        We have used the Tikhonov regularization with regularization parameter $\alpha=10^{-2}$.
        The recovered data get worse as the direction moves further away from the measurable directions.
        A fast and stable method for solving the ill-posed equations is highly desired.
  \item Of greater practical importance would be the case that {\em limited-aperture} is not only for observation directions, but also the incident directions.
  % In particular, a realistic case is to consider the back-scattering {\em limited-aperture} data.
  \item It would be interesting and useful to consider the buried objects, where the measurements are only available in the upper half space.
   % The data for the inverse problem is available over a limited aperture, which implies that the solution of the inverse problem will be degraded compared to situations in which data can be gathered on a sphere containing the object.
\end{itemize}

\section*{Acknowledgements}

The research of X. Liu is supported by the Youth Innovation Promotion Association CAS and the NNSF of China under grant 11571355.
The research of J. Sun is partially supported by NSF DMS-1521555 and NSFC Grant (No. 11771068).

\end{document}